\documentclass[reqno, letterpaper, oneside]{article}
\usepackage{setspace}

\RequirePackage{ifthen}
\RequirePackage{ifpdf}

\ifthenelse{\boolean{pdf}}{
	
} { 
  
}

\usepackage{amsmath,amsthm,amssymb}
\usepackage{bbm}
\usepackage{geometry}

\newtheorem{theorem}{Theorem}
\newtheorem{lemma}[theorem]{Lemma}

\newtheorem{remark}[theorem]{Remark}

\newcommand{\cA}{{\mathcal A}}
\newcommand{\cB}{{\mathcal B}}

\newcommand{\cD}{{\mathcal D}}
\newcommand{\cE}{{\mathcal E}}
\newcommand{\cF}{{\mathcal F}}
\newcommand{\cG}{{\mathcal G}}
\newcommand{\cJ}{{\mathcal J}}
\newcommand{\cH}{{\mathcal H}}
\newcommand{\cI}{{\mathcal I}}
\newcommand{\cK}{{\mathcal K}}

\newcommand{\cM}{{\mathcal M}}

\newcommand{\cT}{{\mathcal T}}

\newcommand{\fX}{{\mathfrak X}}

\newcommand{\fS}{{\mathfrak S}}

\newcommand{\tr}{\tilde{r}}

\newcommand{\E}{{\mathbb E}}

\renewcommand{\Pr}{{\mathbb P}}
\newcommand\Var{\operatorname{Var}}

\newcommand\Bin{\operatorname{Bin}}

\newcommand\NN{\mathbb{N}}

\newcommand{\indic}[1]{\mathbbm{1}_{\{{#1}\}}}

\newcommand\floor[1]{\lfloor {#1} \rfloor}

\newcommand\ceil[1]{\lceil {#1} \rceil}

\renewcommand{\epsilon}{\varepsilon}
\newcommand{\eps}{\varepsilon}

\newcommand\noproof{\qed}

\begin{document}

\title{Upper tails for arithmetic progressions in random subsets} 
\author{Lutz Warnke%
\thanks{Department of Pure Mathematics and Mathematical Statistics, University of Cambridge, 
Wilberforce Road, Cambridge CB3 0WB, UK.
E-mail: {\tt L.Warnke@dpmms.cam.ac.uk}.}}
\date{July 10, 2013; revised March 14, 2016} 
\maketitle

\begin{abstract}
We study the upper tail of the number of arithmetic progressions of a given length in a random subset of $\{1, \ldots, n\}$, establishing exponential bounds which are best possible up to constant factors in the exponent. 
The proof also extends to Schur triples, and, more generally, to the number of edges in random induced subhypergraphs of `almost linear' $k$-uniform hypergraphs. 
\end{abstract}

\section{Introduction}\label{sec:intro}
What is the (typical) behaviour of a given function depending on many independent random variables~$\xi_j$? 
This fundamental concentration-of-measure question is of great interest in 
various areas of pure and applied mathematics, including functional analysis, 
statistical mechanics, and theoretical computer science. 
In applications, concentration inequalities are particularly important: these quantify 
random fluctuations of $X=f(\xi_1, \ldots,\xi_n)$ by bounding the probability that $X$ deviates significantly from its mean~$\E X$. 
During the last decades a wide variety of different methods for proving such inequalities have been developed  
(see, e.g.,~\cite{Ledoux,DubPan,BLM}),  
including martingale based methods~\cite{McDiarmid1989,KimVu2000}, Talagrand's methodology~\cite{Talagrand1995}, combinatorial approaches~\cite{DL}, 
 and information theoretic methods~\cite{Dembo1997,BLM2003}.

Despite this large body of work, in concrete applications  our understanding is often still far from satisfactory --  
even if we restrict our attention to the important case where~$X$ is a sum of (dependent) indicator variables and $\xi_j \in \{0,1\}$. 
For example, in probabilistic combinatorics the random variable~$X$ often counts objects, 
for instance the number of certain subgraphs in random graphs. 
Here Janson's and Suen's inequalities~\cite{Janson,JSuen,JW,RW} usually give sharp estimates for the lower tail $\Pr(X \le (1-\eps)\E X)$. 
In contrast, obtaining tight estimates for $\Pr(X \ge (1+\eps) \E X)$ 
is more delicate, and this `upper tail problem' is well-known to be a 
technical challenge (see, e.g.,~\cite{UT,UTAP}). 

In fact, in many such counting problems each indicator variable depends only on a few 
$\xi_j$, in which case~$X$ has a special structure: it is a low-degree polynomial of independent Bernoulli random variables.  
With this in mind, it is surprising that, 
despite intensive research of Kim and Vu~\cite{KimVu2000,Vu2002} and many others (see, e.g.,~\cite{DL,SSa,TBD,Ledoux,DubPan,BLM}), 
there is no concentration inequality that routinely gives the `correct' upper tail behaviour in these basic situations. 
Consequently the investigation of these and related problems is an important issue 
-- not only from an applications point of view, but also as a question in concentration-of-measure.

In this context, Janson, Oleszkiewicz and Ruci{\'n}ski~\cite{UTSG} developed in 2002 a moment-based method 
that, for subgraph counts in random graphs, gives estimates for $\Pr(X \ge (1+\eps) \E X)$ which are best possible up to logarithmic factors in the exponent. 
Subsequently, Janson and Ruci{\'n}ski~\cite{UTAP} extended this technique so that it also gives comparable 
estimates for arithmetic progressions in random subsets. 
To be more concrete, given $k \ge 3$, let $X$ be the number of arithmetic progressions of length $k$ in $[n]_p$, 
the random subset of $[n]=\{1, \ldots, n\}$ where each element is included independently with probability $p$.  
In~\cite{UTAP} it was shown that for essentially all $p$ and $\eps >0$ of interest we have 
\begin{equation}\label{eq:oldUT}
\exp\Bigl(-C(\eps,k) \sqrt{\E X} \log(1/p)\Bigr) \le \Pr(X \ge (1+\eps) \E X) \le \exp\Bigl(-c(\eps,k) \sqrt{\E X}\Bigr) ,
\end{equation}
determining, as in~\cite{UTSG}, the upper tail up to a factor of $O(\log(1/p))$ in the exponent for constant~$\eps$. 
The problem of closing this logarithmic gap in the approach of Janson et al.~\cite{UTSG,UTAP} has remained open for 
several years, and only very recently have there been some breakthroughs 
by Chatterjee~\cite{K3TailCh} and DeMarco and Kahn~\cite{K3TailDK,KkTailDK} 
for certain subgraph counts.

In this paper we solve the upper tail problem for a wide class of random variables, including arithmetic progressions and Schur triples, 
by establishing upper and lower bounds which match up to constant factors in the exponent. 
For simplicity, we first consider the special case of arithmetic progressions (in Section~\ref{sec:setup} we turn to the general results).  
In particular, \eqref{eq:thm:APp} below shows that 
$\log \Pr(X \ge (1+\eps)\E X) = -\Theta(\min\{\E X,\sqrt{\E X} \log (1/p)\})$ 
for constant~$\eps$, closing the $\log(1/p)$ gap that was present until now.
\begin{theorem}\label{thm:APp}
Given $k \ge 3$, let $X=X_{k,n,p}$ be the number of arithmetic progressions of length $k$ in $[n]_p$. 
Set $\mu=\E X$. 
There are $n_0,b,B > 0$ (depending only on $k$) such that for all $n \ge n_0$, $p \in (0,1]$ and $\eps > 0$ we have 
\begin{equation}\label{eq:thm:APp}
\indic{1 \le (1+\eps)\mu \le X_{k,n,1}}\exp\Bigl(-C(\eps)\Phi\Bigr) \le \Pr(X \ge (1+\eps)\mu) \le \exp\Bigl(-c(\eps)\Phi\Bigr),
\end{equation}
where $\Phi=\min\bigl\{\mu,\sqrt{\mu} \log (1/p)\bigr\}$, $c(\eps)=b \min\{\eps^3,\eps^{1/2}\}$ and $C(\eps)=B \max\{1,\eps^2\}$. 
\end{theorem}
Note that $\mu=\E X = \Theta(n^2p^k)$, and that $p$ and $\eps$ may depend on $n$ (we do not assume $n \ge n_0(\eps)$, $\eps=\Theta(1)$ or $p \ge n^{-2/k}$, which are common in this context). 
The additional condition $(1+\eps)\mu \le X_{k,n,1}$ assumed for the lower bound is necessary (and also implies $p \le (1+\eps)^{-1/k} < 1$); otherwise $X\ge (1+\eps)\mu$ is impossible. The condition $(1+\eps)\mu \ge 1$, which holds automatically under common assumptions such as $\mu=\omega(1)$ or $\mu \ge 1$, is natural; otherwise $\Pr(X \ge (1+\eps) \mu)=\Pr(X \ge 1)$. 
The form of the exponent in \eqref{eq:thm:APp} can be motivated as follows. 
Since an interval $[m]=\{1, \ldots, m\}$ contains $\Theta(m^2)$ arithmetic progressions of length $k$, for suitable $m=\Theta(\sqrt{\mu})$ we have $\Pr(X \ge 2\mu) \ge \Pr([m] \subseteq [n]_p) = p^{\Theta(\sqrt{\mu})}=e^{-\Theta(\sqrt{\mu}\log(1/p))}$. 
Moreover, for small $p$ (say, $p = n^{-2/k}$) we expect that $X$ is approximately Poisson, which suggests $\Pr(X \ge 2\mu) \approx e^{-\Theta(\mu)}$. Theorem~\ref{thm:APp} essentially states that the larger of these bounds determines the decay of the upper tail for constant $\eps$.

A weakness of Theorem~\ref{thm:APp} is that is does not guarantee a similar dependence of $c(\eps)$ and $C(\eps)$ on $\eps$. 
Although results of this form (see, e.g.,~\cite{K3TailCh,KkTailDK,K3TailDK,UTSG}) are the widely accepted standard  for the `infamous' upper tail problem~\cite{UT}, here we go much further. 
Our next result establishes, over a wide range of the parameters, the dependence of the upper tail on $\eps$, up to constants (that are independent of~$\eps$). 
In the language of large deviations, \eqref{eq:thm:AP} below determines, for $p$ bounded away from one, the order of magnitude of the \emph{large deviation rate function} $\log \Pr(X \ge (1+\eps)\E X)$ for all $\eps \ge n^{-\alpha}$ of interest. 
\begin{theorem}\label{thm:AP}
Given $k \ge 3$, let $X=X_{k,n,p}$ be the number of arithmetic progressions of length $k$ in $[n]_p$. 
Set $\mu=\E X$ and $\varphi(x)=(1+x)\log(1+x)-x$. 
Given $\gamma \in (0,1)$, there are $n_0,\alpha > 1/(6k)$ (depending only on $k$) and $c, C>0$ (depending only on $\gamma,k$) such that for all $n\ge n_0$, $p\in (0,1-\gamma]$ and $\eps \ge n^{-\alpha}$ satisfying $\Phi(\eps) \ge 1$ we have 
\begin{equation}\label{eq:thm:AP}
\indic{1 \le (1+\eps)\mu \le X_{k,n,1}}\exp\Bigl(- C\Phi(\eps)\Bigr) \le  \Pr(X \ge (1+\eps)\mu)  \le \exp\Bigl(- c\Phi(\eps)\Bigr) ,
\end{equation}
where $\Phi(\eps)=\min\bigl\{\varphi(\eps)\mu^2/\Var X ,\sqrt{\eps\mu} \log (1/p)\bigr\}$. 
\end{theorem}
It is not hard to check that $\Var X=\Theta(\mu(1+np^{k-1}))$ for $p$ bounded away from one (see, e.g., Example~3.2 and Lemma~3.5 in~\cite{JLR}). 
Note that the condition $\Phi(\eps) \ge 1$ is natural since our focus is on exponentially small probabilities. 
The function $\varphi(x)$  appears in standard Chernoff bounds; it satisfies $\varphi(x)=\Theta(x \log(1+x))$ for $x \ge 0$, so that $\varphi(x)=\Theta(x^2)$ as $x \to 0$. 
The proof of Theorem~\ref{thm:AP} shows that the form of the exponent in \eqref{eq:thm:AP} is determined by Normal approximation considerations (the $\varphi(\eps)\mu^2/\Var X$ term) and the interval clustering idea (the $\sqrt{\eps\mu} \log (1/p)$ term). 
The sharp estimates of Theorem~\ref{thm:AP} are conceptually quite different from previous work on the upper tail problem. 
Indeed, somewhat related work for subgraph counts in the binomial random graph $G_{n,p}$ (which aims to determine the precise constants in the exponent as $n \to \infty$, see, e.g.,~\cite{CD2010,CV2011,LZ2012}) focuses on the case where $\eps$ is constant and $p$ is large (with $p = \Theta(1)$ or $p \ge n^{-\delta}$). 
In fact, for moderately large $p$, our next result \emph{completely} resolves the qualitative behaviour of the upper tail. 
\begin{theorem}\label{thm:APt}
Given $k \ge 3$, let $X=X_{k,n,p}$ be the number of arithmetic progressions of length $k$ in $[n]_p$. 
Set $\mu=\E X$. 
Given $\gamma \in (0,1)$, there are $n_0>0$ (depending only on $k$) and $c, C>0$ (depending only on $\gamma,k$) such that for all $n\ge n_0$, $(\log n)^{1/(k-1)}n^{-1/(k-1)} \le p \le 1-\gamma$ and $t \ge \sqrt{\Var X}$ we have 
\begin{equation}\label{eq:thm:APt}
\indic{\mu + t \le X_{k,n,1}}\exp\Bigl(- C\Psi(t)\Bigr) \le  \Pr(X \ge \mu + t)  \le \exp\Bigl(- c\Psi(t)\Bigr) ,
\end{equation}
where $\Psi(t)=\min\bigl\{t^2/\Var X ,\sqrt{t} \log (1/p)\bigr\}$. 
\end{theorem}
Finally, as the reader can guess, in Theorem~\ref{thm:AP} and~\ref{thm:APt} various conditions (for $\eps$ and $p$) are not best possible. 
However, for ease of exposition we defer more precise results to the next section, where we state our more general tail estimates 
(which include Theorems~\ref{thm:APp}--\ref{thm:APt} as special cases or corollaries). 
Here we just mention that there is a tradeoff between $p$ and $t=\eps \mu$ in Theorem~\ref{thm:AP} and~\ref{thm:APt}. 
Indeed, Theorem~\ref{thm:AP} works for \emph{all} $0 < p \le 1-\gamma$, but~\eqref{eq:thm:AP} is restricted to deviations of form~$\eps \ge n^{-\alpha}$ (for some fixed $\alpha >0$). 
By contrast, Theorem~\ref{thm:APt} requires $n^{-1/(k-1)+o(1)} \le p \le 1-\gamma$, but~\eqref{eq:thm:APt} applies to essentially \emph{all} exponentially small deviations~$t>0$ (note that $\Psi(t) \le 1$ for $t \le \sqrt{\Var X}$).

\subsection{Counting edges of random induced subhypergraphs}\label{sec:setup} 
In this section we present the main results of this paper, Theorem~\ref{thm:HGp} and~\ref{thm:HG}, 
which resolve the upper tail problem (up to constant factors in the exponent) for a large class of 
random variables, including arithmetic progressions and Schur triples. 
We shall phrase our results in the language of random induced subhypergraphs. 
More precisely, given a $k$-uniform hypergraph $\cH$ with vertex set $V(\cH)$, let $V_p(\cH)$ be the random subset of~$V(\cH)$ where each vertex is included independently with probability $p$. 
Define $\cH_p=\cH[V_p(\cH)]$ and 
\begin{equation*}\label{X}
X= e(\cH_p) ,
\end{equation*}
so that $X$ counts the number of edges induced by $V_p(\cH)$. 
Note that $\E X = e(\cH)p^k$. 
Random variables of this form occur frequently in probabilistic combinatorics (see, e.g,~\cite{RR1994,UT,Schacht2009,FRS2010,WG2012,W2014}),  
and, in the setting of Theorems~\ref{thm:APp}--\ref{thm:APt}, the edges of $\cH=\cH_n$ are all $k$-subsets $\{x_1, \ldots, x_k\} \subseteq [n]=V(\cH)$ forming an arithmetic progression of length~$k$. 
To state our results, we define 
\[
\Delta_j(\cH) = \max_{S \subseteq V(\cH): |S|=j} |\{f \in \cH: \: S \subseteq f\}| , 
\]
which for $j \in \{1,2\}$ corresponds to the maximum degree and codegree of~$\cH$, respectively. 
The main examples of~\cite{UTAP} concern $k$-uniform hypergraphs $\cH=\cH_n$ with $v(\cH)=n$ vertices and $e(\cH)=\Theta(n^2)$ edges that are \emph{almost linear}, i.e., with $\Delta_2(\cH) = O(1)$, and satisfy property $\fX(\cH,D,(1+\eps)\mu)$ with $D=\Theta(1)$, where 
\begin{equation}\label{P}
\text{$\fX(\cH,D,x)$: there exists $W \subseteq V(\cH)$ with $|W| \le D \max\{\sqrt{x},1\}$ and $e(\cH[W]) \ge x$.} 
\end{equation}
Note that $\cH=\cH_n$ encoding $k$-term arithmetic progressions in $[n]$ is of this form (see also Remark~\ref{rem:HGp} below). 
Under the aforementioned conditions, 
Janson and Ruci{\'n}ski~\cite{UTAP} proved that the upper tail of $X=e(\cH_p)$ is of type~\eqref{eq:oldUT}, leaving a $\log(1/p)$ gap between the upper and lower bounds for constant $\eps$ (see Theorem~2.1 in~\cite{UTAP} with~$q=2$). 
The following theorem rectifies this issue, by closing the gap. 
\begin{theorem}\label{thm:HGp}
Given $k \ge 3$, $a>0$ and $D \ge 1$, suppose that $\cH=\cH_n$ is a $k$-uniform hypergraph satisfying $v(\cH) \le D n$, $e(\cH) \ge an^2$ and $\Delta_{2}(\cH) \le D$. 
Let $X=e(\cH_p)$ and $\mu=\E X$. 
There are  $n_0,b,B > 0$ (depending only on $k,a,D$) such that for all $n \ge n_0$, $p \in (0,1]$ and $\eps > 0$ we have, with $c(\eps)=b \min\{\eps^3,\eps^{1/2}\}$, 
\begin{equation}\label{eq:thm:HGp:UB}
\Pr(X \ge (1+\eps)\mu) \le \exp\Bigl(-c(\eps)\min\bigl\{\mu, \: \sqrt{\mu}\log(e/p)\bigr\}\Bigr).
\end{equation}
If, in addition, $\fX(\cH,D,(1+\eps)\mu)$ and $(1+\eps)\mu \ge 1$ hold, then we have, with $C(\eps)=B \max\{1,\eps^2\}$, 
\begin{equation}\label{eq:thm:HGp:LB}
\Pr(X \ge (1+\eps)\mu) \ge \exp\Bigl(- C(\eps)\min\bigl\{\mu, \: \sqrt{\mu}\log(1/p)\bigr\}\Bigr). 
\end{equation}
\end{theorem}
\begin{remark}\label{rem:HGp}
In many applications $\fX(\cH,D,x)$ holds automatically for all $x \le e(\cH)$. 
Indeed, often we consider sequences $(\cH_n)_{n \in \NN}$ of hypergraphs satisfying $e(\cH_{n} \cap \cH_{m}) \ge \beta e(\cH_{m})$ for all $n \ge m \ge n_0 $, where $\beta \in (0,1]$ and $n_0 \ge 1$ are constants ($\beta=1$ for monotone sequences, where $\cH_n \subseteq \cH_{n+1}$). 
Then $\fX(\cH_n,D',x)$ follows (by increasing $D$) from $v(\cH_m) \le D m$ and $e(\cH_m) \ge am^2$ 
for $m=\min\{r,n\}$ and suitable $r=\Theta(\max\{\sqrt{x},1\})$. 
\end{remark}
Note that $a n^2 \le e(\cH) \le (v(\cH))^2 \Delta_2(\cH) \le D^3 n^2$, so $\mu=\Theta(n^2p^k)$. 
For \eqref{eq:thm:HGp:LB}, the necessary condition $(1+\eps)\mu \le e(\cH)$ usually entails $\fX(\cH,D,(1+\eps)\mu)$  by Remark~\ref{rem:HGp}, and, as discussed, $(1+\eps)\mu \ge 1$ is very natural (in fact, usually vacuous). 
The assumption $k \ge 3$ is also necessary. 
Indeed, for a concrete counterexample with $k=2$, let $\cH=\cH_n$ contain all pairs $\{x,y\} \subseteq [n]$. 
Since $|[n]_p|$ has a binomial distribution, using $X=e(\cH_p)=\binom{|[n]_p|}{2} \approx |[n]_p|^2/2$ it is 
not difficult to see that $\log \Pr(X \ge (1+\eps) \E X)=-\Theta(\sqrt{\E X})$ for constant~$\eps$ (so there is no extra logarithmic factor).

Turning to applications, using Remark~\ref{rem:HGp} it is easy to see that Theorem~\ref{thm:HGp} applies to the number of arithmetic progressions of length $k$ in $[n]_p$, and so implies Theorem~\ref{thm:APp}. The assumptions of Theorem~\ref{thm:HGp} are also satisfied by Schur triples, which are classical objects 
in Number theory and Ramsey theory (see, e.g.,~\cite{Gr,Sa} and \cite{GRR,Schacht2009}): 
in this case $\cH=\cH_n$ contains all $3$-element subsets $\{x,y,z\} \subseteq [n]$ satisfying $x+y=z$. 
A similar remark applies to the more general notion of $\ell$-sums (studied, e.g., in~\cite{BHKLS,RZ12}), where the $3$-element subsets $\{x,y,z\} \subseteq [n]$ satisfy $x+y=\ell z$. 
Finally, the arguments in Section~2.1 of~\cite{UTAP} reveal that Theorem~\ref{thm:HGp} also applies to the number of integer solutions of certain homogeneous linear systems of equations with rank~$k-2$.

While results similar to Theorem~\ref{thm:HGp} (with constants $c,C$ depending on $\eps$) are usually already considered satisfactory, in this paper we obtain much more precise estimates. Indeed, with Theorem~\ref{thm:HG} below we recover, in a very wide range, the dependence of the upper tail on $t=\eps \mu$ (up to constants). 
Theorem~\ref{thm:HG} looks hard to digest, so we will now spend some time motivating and explaining it. 
As a warm-up, let us first informally discuss the asymptotic form of its upper tail estimates for $X=e(\cH_p)$. 
In particular, since our focus is on exponentially decaying probabilities, in \eqref{eq:thm:HG:UB} and \eqref{eq:thm:HG:LB} below the multiplicative factors of~$1+n^{-1}$ and~$d$ are usually negligible (i.e., can be removed by adjusting the constants~$c,C$). 
Hence, assuming $n^{-2/k}(\log n)^{2/k} \le p \le 1/2$ and $t \ge \sqrt{\Var X}$, say, via Remarks~\ref{rem:Var}--\ref{rem:HG} the form of~\eqref{eq:thm:HG:UB}--\eqref{eq:thm:HG:LB} eventually 
simplifies~to 
\begin{equation}\label{eq:thm:HG}
\log \Pr(X \ge \mu+t) = -\Theta\biggl(\min\biggl\{\frac{t^2}{\Var X}, \: \sqrt{t}\log(1/p)\biggr\}\biggr).
\end{equation}
With this in mind, Theorem~\ref{thm:HG} essentially states that 
the upper tail of $X=e(\cH_p)$ is either of sub-Gaussian type $\exp\bigl(-ct^2/\Var X\bigr)$ or of `clustered' type $\exp\bigl(-c\sqrt{t}\log(1/p)\bigr)$, and that the transition between the two  
happens roughly for~$t$ around~$(\Var X)^{2/3}$.  
In this context the upper bound~\eqref{eq:thm:HG:UB} of Theorem~\ref{thm:HG} is very satisfactory. 
Namely, it holds via~(a) for \emph{all} $t>0$ unless~$p$ is close to $p_0=n^{-1/(k-1)}$, in which case~\eqref{eq:thm:HG:UB} still holds for~$t \ge (\Var X)^{2/3} (\log n)^{4/3}$ via~(b). 
In words, our upper bound~\eqref{eq:thm:HG:UB} recovers the qualitative behaviour of the upper tail for all $t>0$, 
unless~$p$ is in a tiny exceptional interval around~$p_0$ 
(where we basically only miss the sub-Gaussian regime). 
\begin{theorem}\label{thm:HG}
Given $k \ge 3$, $a>0$ and $D \ge 1$, suppose that $\cH=\cH_n$ is a $k$-uniform hypergraph satisfying $v(\cH) \le D n$, $e(\cH) \ge an^2$ and $\Delta_{2}(\cH) \le D$. 
Let $X=e(\cH_p)$, $\mu=\E X$, $\Lambda = \mu (1+np^{k-1})$ and $\varphi(x)=(1+x)\log(1+x)-x$. 
Given $\gamma \in (0,1)$, there are $n_0>0$ (depending only on $k,a,D$) as well as $c,C,d>0$ and $\lambda \ge 1$ (depending only on $\gamma,k,a,D$) such that for all $n \ge n_0$, $p \in (0,1]$ and $t > 0$ the following holds.  
 If one of 
{\vspace{-0.5em}\begin{enumerate}
\itemindent 0.75em \itemsep 0.125em \parskip 0em  \partopsep=0pt \parsep 0em 
 \item[(a)] $p \not\in \bigl(n^{-1/(k-1)-\gamma},\gamma n^{-1/(k-1)}(\log n)^{1/(k-1)}\bigr)$, or
 \item[(b)] $t \ge \gamma \min\{(\Var X)^{2/3},\mu^{2/3}\}(\log n)^{4/3}$, or 
 \item[(c)] $t \ge \mu p^{(k-2)/3-\gamma}$. 
 \end{enumerate}\vspace{-0.5em}\noindent}
holds, then we have the upper bound 
\begin{equation}\label{eq:thm:HG:UB}
\Pr(X \ge \mu+t) \le (1+n^{-1}) \exp\Bigl(-c\min\bigl\{\varphi(t/\mu) \mu^2/\Lambda, \: \sqrt{t}\log(e/p)\bigr\}\Bigr).
\end{equation}
Furthermore, if one of 
{\vspace{-0.5em}\begin{enumerate}
\itemindent 0.75em \itemsep 0.125em \parskip 0em  \partopsep=0pt \parsep 0em 
 \item[(i)] $p \le n^{-2/(k+1/3)}$, or 
 \item[(ii)] $t \ge \min\{(\Var X)^{2/3},\mu^{2/3}\}(\log n)^{2/3}$ and $p \le n^{-1/(k-1)}\log n$, or
 \item[(iii)] $t \ge \min\{\sqrt{\Var X},\sqrt{\Lambda}\}$ and $\gamma n^{-1/(k-1)} \le p \le 1-\gamma$.
 \end{enumerate}\vspace{-0.5em}\noindent}
holds, then $\fX(\cH,D,\min\{\lambda t,\mu+t\})$ and $\mu + t \ge 1$ imply the lower bound 
\begin{equation}\label{eq:thm:HG:LB}
\Pr(X \ge \mu + t) \ge d \exp\Bigl(-C\min\bigl\{\varphi(t/\mu) \mu^2/\Lambda, \: \sqrt{t}\log(1/p) \bigr\}\Bigr) .
\end{equation}
\end{theorem}
\begin{remark}\label{rem:Var}
It is routine to check that $\Var X = \Theta\bigl((1-p)\Lambda\bigr)$, where the implicit constants depend only on $k,a,D$ (analogously to, e.g., Example~3.2 and Lemma~3.5 in~\cite{JLR}). In particular, $\Lambda = \Theta(\Var X)$ holds whenever~$p$ is bounded away from one. 
\end{remark}
\begin{remark}\label{rem:HG}
If $p \ge \gamma n^{-2/k}(\log n)^{2/k}$ or $t \le \mu$, then \eqref{eq:thm:HG:UB}--\eqref{eq:thm:HG:LB} hold with $\varphi(t/\mu) \mu^2/\Lambda$ replaced by $t^2/\Lambda$. 
\end{remark}
In the above assumptions~(a)--(c) and~(i)--(iii), the use of~$\mu$ and~$\Lambda$ is convenient for applications (see, e.g.,~\eqref{eq:thm:t} below), while~$\Var X$ seems more insightful from a conceptual point of view. 
In particular, since we are interested in exponentially small probabilities, by central limit theorem considerations a natural target assumption is~$t \ge \sqrt{\Var X}$, say. 
We now discuss the lower bound~\eqref{eq:thm:HG:LB} of Theorem~\ref{thm:HG}, which tends to have fewer applications. Indeed, for our purposes~\eqref{eq:thm:HG:LB} is mainly important from a concentration-of-measure perspective, since it rigorously \emph{proves} that our upper bound~\eqref{eq:thm:HG:UB} is sharp in a wide range. 
In view of~(i)+(iii), our lower bound~\eqref{eq:thm:HG:LB} only falls short of the target assumption $t \ge \sqrt{\Var X}$ 
for $p \in \bigl(n^{-2/(k+1/3)},n^{-1/(k-1)}\bigr)$, where  $t \ge (\Var X)^{2/3}(\log n)^{2/3}$ suffices by~(ii).  
Perhaps surprisingly, these gaps are solely due to lacking lower bounds of sub-Gaussian type (note that the variance undergoes a transition around $p_0=n^{-1/(k-1)}$ by Remark~\ref{rem:Var}), which until now have been widely ignored in the upper tail literature (see, e.g.,~\cite{Vu2002,WG2012}). 
Here our current approaches 
seem not strong enough to work for all relevant~$p$ and~$t$. 
We leave it as an open problem to develop a generic method for obtaining suitable sub-Gaussian type lower bounds (see Section~\ref{sec:LB:disj}).  
Finally, we also conjecture that the upper tail estimates~\eqref{eq:thm:HG:UB}--\eqref{eq:thm:HG:LB} 
remain valid for all $p \in (0,1-\gamma]$ and $t \ge \sqrt{\Var X}$.

Turning to the remaining applications stated in the introduction,  
Theorem~\ref{thm:APt} for arithmetic progressions follows easily by combining (a)+(iii) of Theorem~\ref{thm:HG} with Remarks~\ref{rem:HGp}, \ref{rem:Var} and~\ref{rem:HG}. 
For Theorem~\ref{thm:AP} we use that, modulo obvious assumptions, the tail estimates \eqref{eq:thm:HG:UB}--\eqref{eq:thm:HG:LB} \emph{both} apply if $t>0$ satisfies, say, 
\begin{equation}\label{eq:thm:t}
t \ge \begin{cases}
		0, & ~~\text{if $0 < p \le n^{-2/(k+1/3)}$}, \\
		\mu^{2/3}(\log n)^{4/3}, & ~~\text{if $n^{-2/(k+1/3)} < p < n^{-1/(k-1)}(\log n)^{1/(k-1)}$},\\
		\sqrt{\Lambda}, & ~~\text{if $n^{-1/(k-1)}(\log n)^{1/(k-1)} \le p \le 1-\gamma$}.
	\end{cases}
\end{equation}
(Using~(a)+(i) for $p \le n^{-2/(k+1/3)}$, (b)+(ii) for larger $p < n^{-1/(k-1)}(\log n)^{1/(k-1)}$, and (a)+(iii) otherwise.) 
As $\mu \ge a n^2p^k$ and $\Lambda = \mu (1+np^{k-1})$, a short calculation reveals that, say, $t \ge \mu n^{-1/(5k+1)}$ implies \eqref{eq:thm:t} for all $n \ge n_0(k,a)$ and $p \in (0,1]$. 
Hence, using Remarks~\ref{rem:HGp} and~\ref{rem:Var}, inequality~\eqref{eq:thm:AP} of Theorem~\ref{thm:AP} follows.

The proofs of the upper and lower bounds of Theorem~\ref{thm:HGp} and~\ref{thm:HG} are based on completely different techniques. 
For the upper bounds~\eqref{eq:thm:HGp:UB} and~\eqref{eq:thm:HG:UB}, the most important ingredients are two new concentration inequalities of Chernoff-type, which we prove in Section~\ref{sec:C}. 
These allow us to combine and extend the combinatorial and probabilistic ideas used in the `deletion method' and the `approximating by a disjoint subfamily' technique of Janson and Ruci{\'n}ski~\cite{DL} and Spencer~\cite{Spencer1990,UT}, respectively. 
The idea of applying the BK-inequality of van den Berg and Kesten~\cite{BK} and Reimer~\cite{BKR} in the context of the `infamous' upper tail problem~\cite{UT}  
may perhaps also be of independent interest. 
For the lower bounds~\eqref{eq:thm:HGp:LB} and~\eqref{eq:thm:HG:LB}, we analyze three different mechanisms that yield deviations of $X=e(\cH_p)$, and with some care (using, e.g., Harris' inequality~\cite{Harris1960} and the Paley--Zygmund inequality) we recover the correct dependence of the exponent on~$t=\eps\mu$. 

The remainder of this paper is organized as follows.
In Section~\ref{sec:C} we introduce our new concentration inequalities, and in Section~\ref{sec:UB} we apply them (together with combinatorial arguments) to prove the upper bounds of Theorem~\ref{thm:HGp} and~\ref{thm:HG}. 
Finally, in Section~\ref{sec:LB} we establish the corresponding lower bounds (and also prove Remark~\ref{rem:HG}).

\section{Concentration inequalities}\label{sec:C}
In this section we introduce our main probabilistic tools: two concentration inequalities which essentially state  
that Chernoff-type upper tail estimates hold  whenever $X$ is 
bounded from above by a sum of random variables with `well-behaved dependencies'. 
They develop ideas of Janson and Ruci{\'n}ski~\cite{DL}, Erd{\H{o}}s and Tetali~\cite{disjoint}, and Spencer~\cite{Spencer1990}, and seem of independent interest. 
On first reading of Theorem~\ref{thm:C} it might be useful to consider the special case where there are independent random variables $(\xi_i)_{i \in \cA}$ such that each $Y_\alpha \in \{0,1\}$ with $\alpha \in \cI$ is a function of $(\xi_i)_{i \in \alpha}$. 
Then, defining $\alpha \sim \beta$ if $\alpha \cap \beta \neq \emptyset$, it is immediate that the independence assumption holds (as $\alpha \not\sim \beta$ implies that $Y_{\alpha}$ and $Y_{\beta}$ depend on disjoint sets of variables~$\xi_i$). Now, consider $X = \sum_{\alpha \in \cI} Y_\alpha$ with $\mu=\E X$, $\cJ = \cI$ and $C = \max_{\beta \in \cI}|\{\alpha \in \cI: \alpha \sim \beta\}|$.  
Then $X=Z_C$, where $\max_{\beta \in \cJ}\sum_{\alpha \in \cJ: \alpha \sim \beta} Y_{\alpha} \le C$ intuitively corresponds to a Lipschitz-like condition. 
With this in mind, part of the power of \eqref{eq:C} is that the exponent scales with $1/C$ (instead of the usual $1/C^2$), and that the Lipschitz condition need not hold deterministically (it suffices if $X \le Z_C$ or $X \approx Z_C$ holds off some exceptional event). 
\begin{theorem}\label{thm:C}
Given a family of non-negative random variables $(Y_{\alpha})_{\alpha \in \cI}$ with $\sum_{\alpha \in \cI} \E Y_{\alpha} \le \mu$, assume that~$\sim$ is a symmetric relation on $\cI$ such that each $Y_{\alpha}$ with $\alpha \in \cI$ is independent of $\{Y_{\beta}: \text{$\beta \in \cI$ and $\beta \not\sim\alpha$}\}$. 
Let $Z_C = \max \sum_{\alpha \in \cJ} Y_\alpha$, where the maximum is taken over all $\cJ \subseteq \cI$ with $\max_{\beta \in \cJ}\sum_{\alpha \in \cJ: \alpha \sim \beta} Y_{\alpha} \le C$.  
Set $\varphi(x)=(1+x)\log(1+x)-x$.  
Then for all $C,t>0$ we have 
\begin{equation}\label{eq:C}
\begin{split}
\Pr(Z_C \ge \mu +t) 
& 
\le \exp\left(-\frac{\varphi(t/\mu)\mu}{C} \right) = 
e^{-\mu/C} \cdot 
\left(\frac{e\mu}{\mu+t}\right)^{(\mu+t)/C} \\ 
& 
\le \min\left\{ \exp\left(-\frac{t^2}{2C(\mu +t/3)}\right), \left(1+\frac{t}{2\mu}\right)^{-t/(2C)}\right\} . 
\end{split}
\end{equation}
\end{theorem}
\begin{remark}\label{rem:C}
Theorem~\ref{thm:C} remains valid after weakening the independence assumption to a form of negative correlation: it suffices if 
$\E (\prod_{i \in [s]} Y_{\alpha_i}) \le \prod_{i \in [s]}  \E Y_{\alpha_i}$ for all $(\alpha_1, \ldots, \alpha_s) \in \cI^s$ satisfying $\alpha_i \not\sim \alpha_j$ for $i \neq j$. 
\end{remark}
Theorem~\ref{thm:C} extends several upper tail inequalities 
discussed in the survey of Janson and Ruci{\'n}ski~\cite{UT}.  
Indeed, consider $X=\sum_{\alpha \in \cI} Y_\alpha$ with $\mu=\E X$ and $\cJ = \cI$. 
For independent $Y_{\alpha} \in [0,1]$ we have $X=Z_1$ (note that $\alpha \sim \alpha$ for non-constant $Y_{\alpha}$), so that \eqref{eq:C} reduces to the classical Chernoff bound, see, e.g., Theorem~2.1 in~\cite{JLR}. 
Similarly, for generic $Y_{\alpha} \in [0,1]$ with dependency graph $\cG=\cG(\cI)$, where distinct $\alpha,\beta \in \cI=V(\cG)$ form an edge if $\alpha \sim \beta$ (cf.\ Section~2.6 in~\cite{UT}), 
we have $X=Z_{\Delta_1(\cG)+1}$. 
Hence~\eqref{eq:C} improves Theorem~5 in~\cite{UT}, which is based on the `breaking into disjoint matchings' technique 
of R{\"o}dl and Ruci{\'n}ski~\cite{RR1994}. 
Furthermore, using $C=t/(2r)$ it is easy to see that Theorem~\ref{thm:C} tightens Theorem~2.1 in~\cite{DL}, i.e., the basic theorem of the `deletion method' of Janson and Ruci{\'n}ski. 
In addition, \eqref{eq:C} extends Lemma~2 in~\cite{UT}, i.e., the main probabilistic ingredient of Spencer's `approximating by a disjoint subfamily' 
technique~\cite{Spencer1990}. 
Theorem~\ref{thm:C} is also related to a concentration inequality of Chatterjee~\cite{K3TailCh}; our assumptions are less technical and subjectively easier to check (e.g., readily implying Proposition~4.1 in~\cite{K3TailCh} via $C=3 \eps \ell np$). 
Remark~\ref{rem:C} is useful in the context of the uniform random graph $G_{n,m}$ (and related uniform models). 
To illustrate this we consider $Y_{\alpha}=\indic{\alpha \subseteq E(G_{n,m})}$ and set $\alpha \sim \beta$ if $\alpha \cap \beta \neq \emptyset$. 
In that case it is well-known (and not hard to check) 
that the negative correlation condition of Remark~\ref{rem:C} holds, demonstrating that Theorem~\ref{thm:C} applies to~$G_{n,m}$. 
\begin{proof}[Proof of Theorem~\ref{thm:C}] 
The proof is based on a variant of the $m$-th factorial moment which `forces independence'. 
In fact, we closely follow Lemma~2.3 in~\cite{DL} and Lemma~2.46 in~\cite{JLR}, but differ in some important details. 
Assume that $m \in \NN$ satisfies $1 \le m \le \ceil{(\mu+t)/C}$. 
For all $\cK \subseteq \cI$ and $s \in \NN$ with $s \ge 1$ we define 
\begin{equation*}\label{eq:Zs}
M_{s}(\cK) = \sideset{}{^*}\sum_{(\alpha_1, \ldots, \alpha_s)\in \cK^s}  \prod_{i \in [s]} Y_{\alpha_i},
\end{equation*} 
where $\sum^{*}_{(\alpha_1, \ldots, \alpha_s)\in \cK^s}$ denotes 
the sum over all tuples $(\alpha_1, \ldots, \alpha_s) \in \cK^s$ satisfying $\alpha_i \not\sim \alpha_j$ for $i \neq j$. 
The key point is that, by construction, the factors $Y_{\alpha_i} \ge 0$ in each term of $M_s(\cK)$ are independent. 
Hence 
\begin{equation}
\label{eq:C:EZ}
 \E M_m(\cI) = \sideset{}{^*}\sum_{(\alpha_1, \ldots, \alpha_m)\in \cI^m}  \E \bigl(\prod_{i \in [m]} Y_{\alpha_i}\bigr) = \sideset{}{^*}\sum_{(\alpha_1, \ldots, \alpha_m)\in \cI^m}  \prod_{i \in [m]} \E Y_{\alpha_i} \le  \Big(\sum_{\alpha \in \cI} \E Y_{\alpha}\Big)^m \le \mu^m .
\end{equation}

Now assume that $Z_C \ge \mu+t$ and $Z_C=\sum_{\alpha \in \cJ}Y_{\alpha}$ hold. 
Note that, by construction, $M_{1}(\cJ) = \sum_{\alpha \in \cJ} Y_\alpha = Z_C \ge \mu+t$. 
Furthermore, by choice of $\cJ$ (see the definition of $Z_C$), for all $(\alpha_1, \ldots, \alpha_s) \in \cJ^s$ we have 
\begin{equation*}\label{eq:C:UB}
\sum_{\substack{\alpha \in \cJ: \alpha \sim \alpha_i\\\text{for some $i \in [s]$}}} Y_\alpha \le \sum_{i \in [s]}\sum_{\alpha \in \cJ: \alpha \sim \alpha_i} Y_{\alpha} \le Cs. 
\end{equation*}
So, for all $s \in \NN$ with $1 \le s < m \le \ceil{(\mu+t)/C}$ it follows that 
\begin{equation}\label{eq:C:induct}
M_{s+1}(\cJ) =  \sideset{}{^*}\sum_{(\alpha_1, \ldots, \alpha_s) \in \cJ^s} \prod_{i \in [s]} Y_{\alpha_i} \cdot \Big(\sum_{\alpha \in \cJ} Y_\alpha-\sum_{\substack{\alpha \in \cJ: \alpha \sim \alpha_i\\\text{for some $i \in [s]$}}} Y_\alpha \Big) \ge  M_{s}(\cJ) \cdot (\mu+t-Cs) ,
\end{equation}
which by induction yields $M_m(\cJ) \ge \prod_{j=0}^{m-1}(\mu+t-Cj)$. 

Combining the above estimates for $M_m(\cI) \ge M_m(\cJ)$ and $\E M_m(\cI)$ with Markov's inequality, we obtain 
\begin{equation}\label{eq:C:Markov}
\Pr(Z_C \ge \mu +t) \le \Pr\Bigl(M_m(\cI) \ge \prod_{j=0}^{m-1}(\mu+t-Cj)\Bigr) \le \prod_{j=0}^{m-1} \frac{\mu}{\mu+t-Cj}.
\end{equation}
Set $m=\ceil{t/C} \ge 1$. 
If $\mu=0$, then $\Pr(Z_C \ge \mu +t) =0$ by \eqref{eq:C:Markov}, and \eqref{eq:C} is trivial, so we henceforth assume $\mu>0$.  
For $0 \le x \le t/C$, the function $f(x)=\log(\mu/(\mu+t-Cx))$ is increasing and satisfies $f(x) \le 0$. 
As $f(t/C)=0$, it follows that $f(j) \le \int_{j}^{\min\{j+1,t/C\}} f(x) dx$ for $0 \le j \le t/C$. 
We deduce 
\[
\log \Pr(Z_C \ge \mu +t)  \le 
\sum_{j=0}^{\ceil{t/C}-1} \log\left(\frac{\mu}{\mu+t-Cj}\right) \le 
\int_{0}^{t/C} \log\left(\frac{\mu}{\mu+t-Cx}\right) dx =: \Psi. 
\]
Using $\log(a/b)=\log a-\log b$, integration yields $\Psi = -\varphi\left(t/\mu\right)\mu/C$. It is well-known that 
\begin{equation}\label{eq:varphi:x1}
\varphi(x) \ge x^2/(2+2x/3)
\end{equation}
for $x \ge 0$ (see, e.g., the proof of Theorem~2.1 in~\cite{JLR}), so $\Psi \le -t^2/\bigl(2C(\mu+t/3)\bigr)$. 
Finally, for $u=t/(2C)$ we have $\Psi = \int_{0}^{t/C} f(x) dx \le \int_{0}^{u} f(x) dx \le u f(u)$, which establishes \eqref{eq:C}. 
\end{proof}
For all integers $x \ge 1$, by formally defining $xC=\mu+t$ and $m=x$ in the above proof (so that $\mu+t-Cj = C(x-j)$ holds), note that inequality~\eqref{eq:C:Markov} and Stirling's formula imply 
\begin{equation}\label{eq:C:ET}
\Pr(Z_C \ge xC) \le \left(\frac{\mu}{C}\right)^x/x! \le \left(\frac{e\mu}{xC}\right)^x / \sqrt{2 \pi x} .
\end{equation}
While this estimate is often weaker than~\eqref{eq:C}, for $C=1$ it extends, in the upper tail context, the so-called `disjointness lemma' of Erd{\H{o}}s and Tetali~\cite{disjoint}, see, e.g., Lemma~8.4.1 in~\cite{AS}. 
In the proof of Theorem~\ref{thm:C}, inequality~\eqref{eq:C:EZ} is the only step in which anything is assumed about the $Y_\alpha$, and independence is used in a limited way: $\E (\prod Y_{\alpha_i}) \le \prod \E (Y_{\alpha_i})$ suffices (in fact, replacing the assumption $\sum \E Y_{\alpha} \le \mu$ with $\sum \lambda_{\alpha} \le \mu$ and $\lambda_{\alpha} \ge 0$, it suffices if $\E (\prod Y_{\alpha_i}) \le \prod \lambda_{\alpha_i}$ holds). 
This suggests that the argument is rather robust, since, e.g., ad-hoc upper bounds for $\E (\prod Y_{\alpha_i})$ are enough to obtain tail inequalities, see the proof of Lemma~4.5 in~\cite{K4free}.
Finally, in~\eqref{eq:C:induct} there is also potential for relaxing $\max_{\beta \in \cJ}\sum_{\alpha \in \cJ: \alpha \sim \beta} Y_{\alpha} \le C$ to an accumulative condition (e.g., replacing~$Cs$ by~$t/2$).

The following variant of Theorem~\ref{thm:C} exploits the BK-inequality~\cite{BK} to further relax the independence assumption. 
Clearly, two events $\cE_1$, $\cE_2$ depending on disjoint sets of independent random variables are independent. 
For our purposes it intuitively suffices if, for each possible outcome $\omega \in \Omega$, we can `certify' the occurrence of $\cE_1$ and $\cE_2$ by disjoint sets of variables (which may depend on $\omega$). 
For $\omega=(\omega_1, \ldots, \omega_M) \in \Omega=\Omega_{1}\times \cdots\times\Omega_M$ and $K \subseteq [M]=\{1, \ldots, M\}$ we write ${\omega|}_K=(\omega_i)_{i \in K}$ and $[\omega]_{K} = \{\omega' \in \Omega: {\omega'|}_K = {\omega|}_K \}$. 
If $[\omega]_{K} \subseteq \cE$, then ${\omega|}_{K}$ is called a \emph{certificate} for the occurrence of the event $\cE$ (in words, $\cE$ occurs on all sample points that agree with $\omega$ restricted to $K$). 
Intuitively speaking, in Theorem~\ref{thm:CD} the random variable $Z$ counts the maximum number of events that `occur disjointly', i.e., have disjoint certificates. 
With this in mind, a key feature of inequalities \eqref{eq:C} and \eqref{eq:C:ET} is that they are dimension-free: they do \emph{not} involve the sizes of the certificates (in contrast to `certificate-based' variants of Talagrand's inequality such as Theorem~2 in~\cite{McDR2006}). 
\begin{theorem}\label{thm:CD}
Given a product space $\Omega=\Omega_{1}\times \cdots\times\Omega_M$, with finite $\Omega_i$, let $(\cE_{\alpha})_{\alpha \in \cI}$ be a family of events with $\sum_{\alpha \in \cI}\Pr(\cE_{\alpha}) \le \mu$. Let $Z=\max |\cJ|$, where the maximum is taken over all $\cJ \subseteq \cI$ for which there are disjoint $K_i \subseteq [M]$ satisfying $[\omega]_{K_i} \subseteq \cE_{\alpha_i}$ for all $\alpha_i \in \cJ$. 
Then~\eqref{eq:C} and~\eqref{eq:C:ET} hold with~$C=1$ and~$Z_1=Z$. 
\end{theorem}
\begin{remark}\label{rem:CD}
Theorem~\ref{thm:CD} remains valid after weakening the product space assumption: restricting to increasing events $\cE_{\alpha} \subseteq \Omega=\{0,1\}^M$, it suffices if $\Pr$ satisfies the BK-inequality~\eqref{eq:BKR} for increasing events (in this case $\square$ is associative, so we may replace~$Z$ by the maximum of~$|\cJ|$ over all $\cJ \subseteq \cI$ for which $\square_{\alpha \in \cJ}\cE_\alpha$ holds). 
\end{remark}
The proof of Theorem~\ref{thm:CD} is based on the BK-inequality, which is a partial converse to Harris' inequality~\cite{Harris1960}. 
Intuitively, $\cA \square \cB$ means that the events $\cA$ and $\cB$ have disjoint certificates. Formally, we define \[
\cA \square \cB = \{\omega \in \Omega: \ \text{there are disjoint $K,L \subseteq [M]$ such that $[\omega]_{K} \subseteq \cA$ and $[\omega]_{L} \subseteq \cB$}\} ,
\]
which need not be associative. 
The general BK-inequality of Reimer~\cite{BKR} states that for any product space $\Omega=\Omega_{1}\times \cdots\times\Omega_M$, with finite $\Omega_i$, the following holds: for any two events $\cA,\cB \subseteq \Omega$ we have 
\begin{equation}\label{eq:BKR}
\Pr(\cA \square \cB) \le \Pr(\cA) \Pr(\cB) .
\end{equation}
\begin{proof}[Proof of Theorem~\ref{thm:CD}] 
The proof uses a $\square$-based variant of the $m$-th moment (inspired by Theorem~\ref{thm:C}). 
For all $m \in \NN$ we define  $\cD(\alpha_1, \ldots, \alpha_m) = ((\cdots (\cE_{\alpha_1} \square\cE_{\alpha_2})\square \cdots)\square\cE_{\alpha_{m-1}}) \square \cE_{\alpha_m}$ and 
\[
M_{m}(\cK) = \sum_{(\alpha_1, \ldots, \alpha_m)\in \cK^m}  \indic{\cD(\alpha_1, \ldots, \alpha_m)}.
\] 
Using the BK-inequality~\eqref{eq:BKR} inductively, we obtain $\Pr(\cD(\alpha_1, \ldots, \alpha_m)) \le \prod_{i \in [m]} \Pr(\cE_{\alpha_i})$. 
So, analogous to~\eqref{eq:C:EZ}, we deduce $\E M_{m}(\cI) \le \mu^m$. 
Now assume that $Z \ge y$ and $Z=|\cJ|$ hold. 
For each $m \le \ceil{y} \le |\cJ|$, by definition of $Z$ we see that $\cD(\alpha_1, \ldots, \alpha_m)$ occurs for all $m$-element subsets $\{\alpha_1, \ldots, \alpha_m\} \subseteq \cJ$. 
Hence 
\[
M_{m}(\cI) \ge M_{m}(\cJ) \ge \binom{|\cJ|}{m}m! \ge \prod_{j=0}^{m-1}(y-j).
\] 
Let $Z_1=Z$ and $C=1$. With $y=\mu+t$, the proof of Theorem~\ref{thm:C} carries over unchanged from \eqref{eq:C:Markov} onwards, and \eqref{eq:C} follows.  
Similarly, with $y=x$, $m=x$ and $\mu+t=x$, \eqref{eq:C:Markov} establishes \eqref{eq:C:ET}.    
\end{proof}
The sufficient condition of Remark~\ref{rem:CD} has recently been established in~\cite{BKkoutn} for $\Pr$ assigning equal probability to all $\omega \in \{0,1\}^M$ with exactly $k$ ones. 
Hence Theorem~\ref{thm:CD} applies to $G_{n,m}$ and related uniform models.

\section{Upper bounds}\label{sec:UB} 
In this section we establish the upper bounds~\eqref{eq:thm:HGp:UB} and~\eqref{eq:thm:HG:UB} of Theorem~\ref{thm:HGp} and~\ref{thm:HG}. 
The executive summary of our 
proof strategy is as follows: using combinatorial arguments we shall approximate $X=e(\cH_p)$ using several `well-behaved' auxiliary random variables, which we in turn estimate by the concentration inequalities of Section~\ref{sec:C}. 
Of course, the actual details are much more involved, and our arguments in fact develop  
combinatorial and probabilistic ideas of 
the `deletion method'~\cite{DL} and the `approximating by a disjoint subfamily' technique~\cite{Spencer1990,UT}.  
We have added a substantial amount of informal discussion and motivation to the remainder of this section, in an attempt to make the underlying ideas and techniques more accessible (the actual proofs could be recorded in a much shorter way). 
For example, in order to milden some of the technical difficulties, we shall not only informally discuss the intriguing $\log(e/p)$ factors in the exponent, 
but also prove~\eqref{eq:thm:HGp:UB} using a simplified version our arguments (instead of proving~\eqref{eq:thm:HGp:UB} and~\eqref{eq:thm:HG:UB} in a unified way).

The remainder of this section is organized as follows. 
In Section~\ref{sec:m} we motivate parts of our proof strategy, and illustrate how logarithmic terms arise in our tail estimates. 
In Section~\ref{sec:CA} we then present our basic proof framework, and establish the upper bound of Theorem~\ref{thm:HGp}. 
Finally, in Section~\ref{sec:CO} we refine the aforementioned framework, and prove the more involved upper bound of Theorem~\ref{thm:HG}.

\subsection{Warming up}\label{sec:m} 
The upper bounds of Theorem~\ref{thm:HGp} and~\ref{thm:HG} involve \emph{exponentially} small probabilities, so error probabilities of form~$o(1)$ are too crude for our purposes (and the proofs require more care). 
In fact, the exponents in~\eqref{eq:thm:HGp:UB} and~\eqref{eq:thm:HG:UB} are fairly involved, and both contain somewhat unusual $\log(e/p)$ terms. 
With these non-standard features in mind, the goals of this informal section are two-fold: (i)~to motivate some details of our upcoming proof strategy, and (ii)~to illustrate the way in which we eventually obtain the $\log(e/p)$ factors.

\subsubsection{Motivation and preliminaries}\label{sec:mot}
Let us start with a basic estimate for the number of induced edges $X=e(\cH_p)$. 
For brevity we set 
\begin{equation*}\label{eq:gammav}
\Gamma_v(\cG) = \{f \in \cG: v \in f\} ,
\end{equation*}
so that $|\Gamma_v(\cH_p)|$ equals the degree of vertex~$v$ in $\cH_p$. 
Clearly, for all $r > 0$ we have 
\begin{equation}\label{eq:PPinduct}
\begin{split}
\Pr(X \ge \mu+t) & \le \Pr(X \ge \mu+t \text{ and } \Delta_1(\cH_p) \le r) + \Pr(\Delta_1(\cH_p) > r) \\
& \le \Pr(X \ge \mu+t \text{ and } \Delta_1(\cH_p) \le r) + \sum_{v \in V(\cH)}\Pr(|\Gamma_v(\cH_p)| >r).
\end{split}
\end{equation}
A similar decomposition forms the basis of the inductive  `deletion method' of Janson and Ruci{\'n}ski~\cite{DL}, see, e.g., Theorem~2.5 and Section~3 in~\cite{DL}.
The inductive approach of Kim and Vu~\cite{KimVu2000} is also based on a related idea, see, e.g., Section~3.2 in~\cite{Vu2002}.

One bottleneck of the above approach~\eqref{eq:PPinduct} is that it relies on a uniform upper bound on the degree of all vertices.  
We shall rectify this issue via the following sparsification strategy (which allows for some vertices with larger degrees): 
we first decrease the maximum degree of~$\cH_p$ by removing some carefully chosen edges, 
and then estimate the number of \emph{remaining} edges via the Chernoff-type tail inequality~Theorem~\ref{thm:C}. 
In other words, our plan is to first apply further combinatorial arguments to~$\cH_p$, before using any probabilistic tail estimates or induction.   
An embryonic version of this idea is contained in the `approximating by a disjoint subfamily' technique of Spencer~\cite{Spencer1990,UT}, but Janson and Ruci{\'n}ski argued in their upper tail survey~\cite{UT} that this technique is `never better' than the `deletion method'~\cite{DL} (see Remark~2 in Section~2.3.4 and Example~7 in Section~3.2 of~\cite{UT}). 
In Sections~\ref{sec:CA}--\ref{sec:CO} we shall, in some sense, crossbred ideas of both approaches to go one step further.

\subsubsection{Extra logarithmic factors in tail estimates?}\label{sec:log} 
Let us illustrate how extra logarithmic factors can arise in our upper tail estimates. 
To this end we shall now have, in the context of Theorem~\ref{thm:HGp}, a heuristic look at the exponential decay of the degrees $|\Gamma_v(\cH)|$. 
Here the key observation is that the dependencies among the edges in $\Gamma_v(\cH_p) \subseteq \Gamma_v(\cH)$ are severely limited by the codegree condition $\Delta_2(\cH) = O(1)$: for every $e \in \Gamma_v(\cH)$ there are only at most $k \Delta_2(\cH) = O(1)$ edges $f \in \Gamma_v(\cH)$ which intersect $e \setminus \{v\}$, i.e., with $(f \cap e) \setminus \{v\} \neq \emptyset$ (because all such~$f$ contain~$v$ and at least one vertex from $e \setminus \{v\}$). 
As~$\cH$ is $k$-uniform, it thus seems plausible that, conditioned on~$v \in V_p(\cH)$,  
the upper tail of~$|\Gamma_v(\cH_p)|$ decays roughly like a binomial random variable $Y \sim \Bin(|\Gamma_v(\cH)|,p^{k-1})$. 
Note that for all positive integers~$x$, we have 
\begin{equation}\label{eq:heur:Binx}
\Pr\bigl(Y \ge x\bigr) \le \binom{|\Gamma_v(\cH)|}{x} p^{(k-1)x} \le \frac{\bigl(|\Gamma_v(\cH)|p^{k-1}\bigr)^{x}}{x!} \le \left(\frac{O(np^{k-1})}{x}\right)^x , 
\end{equation}
where we used $|\Gamma_v(\cH)| \le |V(\cH)| \cdot \Delta_2(\cH)= O(n)$ for the last inequality. 
As expected, the decay of $|\Gamma_v(\cH_p)|$ turns out to be very similar to~\eqref{eq:heur:Binx}. 
Indeed, ignoring a number of technicalities, we later approximately show 
(see~\eqref{eq:Mrh:Phir} in the proof of Lemma~\ref{lem:MrH}) 
that for a certain range of~$x$ we have 
\begin{equation}\label{eq:heur:maxd}
\Pr(\Delta_1(\cH_p) \ge x) \le \sum_{v \in V(\cH)}\Pr(|\Gamma_v(\cH_p)| \ge x) \le \left(\frac{O(np^{k-1})}{x}\right)^{\Theta(x)}  .
\end{equation}
With this in mind, the basic idea for `extra' logarithmic terms is simple: if $x \gg y np^{k-1}$ holds, then \eqref{eq:heur:maxd} suggests $\Pr(\Delta_1(\cH_p) \ge x) \le \exp\bigl(-\Theta(x \log  y)\bigr)$. 
In words, if the deviation~$x$ `overshoots' the expectation $|\Gamma_v(\cH)|p^{k-1} = O(n p^{k-1})$ significantly, then we should win a logarithmic factor in the exponent.

In Sections~\ref{sec:CA}--\ref{sec:CO} we shall exploit the aforementioned `overshooting' phenomenon for a range of different degrees (to intuitively show that there are not too many vertices with high degrees). 
Of course, using this approach we shall eventually need to check a number of technical conditions such as $np^{k-1}/x = O\bigl(p^{\Theta(1)}\bigr)$: 
these are \emph{key} for obtaining the $\log(e/p)$ factors missing in previous work 
of Janson and Ruci{\'n}ski~\cite{UTAP}.

\subsection{Basic proof framework}\label{sec:CA}  
In this section we introduce our basic proof framework (for arbitrary hypergraphs~$\cH$), which seems of independent interest. 
In the combinatorial part we implement the sparsification idea mentioned in Section~\ref{sec:mot}, and essentially show the number of induced edges $X=e(\cH_p)$ can be estimated via two carefully defined auxiliary random variables~$X_r=X_r(\cH_p)$ and~$M_r=M_r(\cH_p)$. 
In the probabilistic part we systematically obtain upper tail estimates for~$X_r$ and~$M_{r}$, by exploiting the Chernoff-type concentration inequalities of Section~\ref{sec:C}. 
Finally, we demonstrate the applicability of this framework by proving the upper bound of Theorem~\ref{thm:HGp}.

Recall that our strategy is to decrease the maximum degree of~$\cH_p$ by removing edges. 
To estimate the upper tail of the remaining edges, we now introduce the following `smooth approximation' of~$X=e(\cH_p)$: 
\begin{equation}\label{def:Xr}
X_r = \max\bigl\{e(\cG) : \ \text{$\cG \subseteq \cH_{p}$ and $\Delta_1(\cG) \le r$} \bigr\} .
\end{equation} 
In words, $X_r=X_r(\cH_p)$ denotes the maximum number of edges in \emph{any} subhypergraph $\cG \subseteq \cH_{p}$ with maximum degree at most~$r$. 
Via Theorem~\ref{thm:C} this `bounded degree' property eventually yields~\eqref{eq:Xr}, i.e, a general upper tail estimate for $X_r$. 
For $\eps=\Theta(1)$ and $k= \Theta(1)$, note that~\eqref{eq:Xr} yields $\Pr(X_r \ge (1+\eps/2)\mu) \le \exp(-\Theta(\mu/r))$. 
\begin{lemma}\label{lem:Xr}
Suppose that $\cH$ satisfies $\max_{f \in \cH} |f| \le k$.   
Set $X=e(\cH_p)$, $\mu = \E X$ and $\varphi(x)=(1+x)\log(1+x)-x$. 
Then, for all $p \in [0,1]$ and $r,t > 0$ we have 
\begin{equation}\label{eq:Xr}
\Pr(X_r \ge \mu+t/2) \le \exp\left(-\frac{\varphi(t/\mu) \mu}{4kr} \right) \le \exp\left(-\frac{\min\{t,t^2/\mu\}}{12kr} \right) . 
\end{equation}
\end{lemma}
The main observation required to deduce Lemma~\ref{lem:Xr} from Theorem~\ref{thm:C} is that every edge $f \in \cG \subseteq \cH$ is incident to at most $k \Delta_1(\cG)$ other edges of~$\cG$. 
This allows us to bring the Lipschitz-like condition of Theorem~\ref{thm:C} into play (with $C=kr$). 
\begin{proof}[Proof of Lemma~\ref{lem:Xr}]
Defining $Y_{f} = \indic{f \subseteq V_p(\cH)}$,  
we have $\sum_{f \in \cH} \E Y_f=\E X=\mu$. 
Set $e \sim f$ if $e \cap f \neq \emptyset$. Hence, by the discussion preceding Theorem~\ref{thm:C}, the independence assumption of Theorem~\ref{thm:C} holds (here the $\xi_i=\indic{i \in V_p(\cH)}$ are independent indicators, so $Y_f = \prod_{i \in f}\xi_i$).  
Observe that for all $f \in \cG \subseteq \cH$ we have  
\[
\sum_{e \in \cG: e \sim f} Y_e \le \sum_{v \in f}\sum_{e \in \cG: v \in e} Y_e \le |f| \cdot \max_{v \in f}|\Gamma_v(\cG)| \le k \Delta_1(\cG). 
\]
Hence, for $C=kr$ we deduce $X_r \le Z_C$, where $Z_C$ is defined as in Theorem~\ref{thm:C} with $\cI=\cH$. So, using \eqref{eq:C},
\[
\Pr(X_r \ge \mu+t/2) \le \Pr(Z_C \ge \mu+t/2) \le \exp\left(-\frac{\varphi\bigl(t/(2\mu)\bigr) \mu}{kr} \right) ,
\]
and it remains to rewrite this estimate. 
Since \eqref{eq:varphi:x1} implies (by distinguishing the cases $x \ge 1$ and $x \le 1$) that
\begin{equation}\label{eq:varphi:x4}
\varphi(x) \ge \min\{x,x^2\}/3,
\end{equation}
we see that \eqref{eq:Xr} follows if $\varphi(t/(2\mu)) \ge \varphi(t/\mu)/4$. 
To sum up, it suffices to prove that
\begin{equation}\label{eq:varphi:x2}
\varphi(x/2) \ge \varphi(x)/4 
\end{equation}
for $x \ge 0$. To this end we consider $f(x) = \varphi(x/2)-\varphi(x)/4$. Now, for $x \ge 0$ we have $4f'(x) = \log\bigl((1+x/2)^2/(1+x)\bigr) \ge 0$, so that $f(x) \ge f(0)=0$, 
completing the proof. 
\end{proof}

Our sparsification strategy intuitively focuses on high-degree vertices (with degree at least~$r$).  
To quantify the number of removed edges, we shall introduce 
the auxiliary variable~$M_r=M_r(\cH_p)$, which essentially counts high-degree vertices with `disjoint certificates' (in the sense of Section~\ref{sec:C}). 
More precisely, we call $S=(v,W)$ an \emph{$r$-star} in $\cG$ if $W = \{f_1, \ldots, f_{\ceil{r}}\} \subseteq \Gamma_v(\cG)$ and $|W|=\ceil{r}$. 
We write $V(S) =  \bigcup_{1 \le i \le {\ceil{r}}} f_i$, which contains all vertices of the $r$-star~$S$. 
Note that $V(S) \subseteq V_p(\cH)$ implies $|\Gamma_v(\cH_p)|\ge \ceil{r}$, i.e., that vertex~$v$ has degree at least~$\ceil{r}$.  
Writing $\cT_r(\cG)$ for the collection of all $r$-stars $S=(v,W)$ in $\cG$, we define 
\begin{equation}\label{def:Mr}
\begin{split}
M_r(\cG) &= \max \bigl\{|\cM| : \ \text{$\cM \subseteq \cT_{r}(\cG)$ and $V(S_1) \cap V(S_2) = \emptyset$ for all distinct $S_1,S_2 \in \cM$}\bigr\} .
\end{split}
\end{equation}
In words, $M_r(\cH_p)$ denotes the size of the largest vertex disjoint collection of $r$-stars in $\cH_p$, i.e., $r$-star matching.
(As indicated earlier, it might be useful to think of $M_r(\cH_p)$ as the maximum number of degree~$\ge r$ vertices that `occur disjointly'.) 
For future reference we note the following basic relation between $\Delta_1(\cH_p)$ and $M_r(\cH_p)$. 
\begin{lemma}\label{obs:MrT}
Given $\cH$, for all $p \in [0,1]$ and $z > 0$ we have $\Pr(\Delta_1(\cH_p) \ge z) = \Pr(M_{z}(\cH_p) \ge 1)$. \noproof 
\end{lemma}
The following combinatorial lemma is at the heart of our basic sparsification strategy: it intuitively relates $X=e(\cH_p)$ with the auxiliary random variables $X_r$ and $M_r(\cH_p)$. 
In fact, inequality~\eqref{eq:approx:basic} below 
is inspired by the main deterministic ingredient of the `approximating by a disjoint subfamily' technique 
(see, e.g., Lemma~3 in~\cite{UT}, which is used to count \emph{vertices} in an auxiliary graph with $V(G)=\cH_p$).  
While Spencer's technique hinges on the fact that disjoint edges are nearly independent (see also~\cite{Spencer1990,disjoint}), here one 
important conceptual difference is that we allow for dependencies, i.e., overlaps of the edges (via $r \ge 2$ in~$X_r$). 
For our applications the crux of~\eqref{eq:approx:basic} is that $X_{r} < (1+\eps/2)\mu$ and $k \ceil{r} M_r(\cH_p) \Delta_1(\cH_p) < \eps\mu/2$ together imply $X < (1+\eps)\mu$. 
\begin{lemma}\label{lem:approx:basic}
Suppose that $\cH$ satisfies $\max_{f \in \cH} |f| \le k$. 
Then, for all $p \in [0,1]$ and $r > 0$ we have 
\begin{equation}\label{eq:approx:basic}
X_{r} \leq X \leq X_{r} + \indic{\Delta_1(\cH_p) > r} k \ceil{r} M_r(\cH_p) \Delta_1(\cH_p) .
\end{equation}
\end{lemma}
The proof idea is simple: if $\cM \subseteq \cT_{r}(\cH_{p})$ attains the maximum in the definition of $M_r(\cH_p)$, then after removing all edges incident to some star $S=(v,W) \in \cM$ we obtain a hypergraph $\cG$ with maximum degree at most $\ceil{r}-1 \le r$ (otherwise we could add another $r$-star to the vertex disjoint collection $\cM$), so $e(\cG) \le X_r$. 
Inequality~\eqref{eq:approx:basic} combines this observation with trivial estimates for the number of removed edges. 
\begin{proof}[Proof of Lemma~\ref{lem:approx:basic}]
The lower bound $X =e(\cH_{p})\geq X_r$ is immediate.  
For the upper bound, note that $X = X_{r}$ whenever $\Delta_1(\cH_p) \le r$, so we may henceforth assume $\Delta_1(\cH_p) > r$. 
We fix some $\cM \subseteq \cT_{\ceil{r}}(\cH_{p})$ which attains the maximum in \eqref{def:Mr}, so $M_r(\cH_p)=|\cM|$. 
We remove all edges from $\cH_p$ which  contain at least one vertex from (the edges of) some $r$-star $S=(v, \{f_1, \ldots, f_{\ceil{r}}\}) \in \cM$, and denote the remaining hypergraph by~$\cG$.
As every edge contains at most $\max_{f \in \cH} |f| \le k$ vertices, 
we removed at most $e(\cH_p)-e(\cG) \le |\cM| \cdot \ceil{r}k \cdot \Delta_1(\cH_{p})$ edges from $\cH_p$. 
Clearly $\Delta_1(\cG) \le \ceil{r}-1 \le r$, because otherwise we could add another $r$-star to~$\cM$ (contradicting maximality). 
Hence $\cG$ contains at most $e(\cG) \le X_r$ edges, and~\eqref{eq:approx:basic} follows.  
\end{proof} 
Next, we shall exploit the disjoint-like structure of $M_r(\cH_p)$ via the BK-inequality based Theorem~\ref{thm:CD}. 
This leads to~\eqref{eq:Mr}, a generic upper tail estimate for the size of the largest $r$-star matching $M_r(\cH_p)$. 
Note that $\Pr(\Delta_1(\cH_p) \ge r) \le \sum_{v \in V(\cH)} \Pr(|\Gamma_v(\cH_p)| \ge \ceil{r}) = \Phi_r$. 
In this paper we mainly have very unlikely degrees in mind, where $\Phi_r \le Q^{-r}$ for some $Q > 1$. 
Then the probability that at least~$y$ of such high-degree vertices (with degree at least~$r$) `occur disjointly' is roughly at most $Q^{-ry}$ by~\eqref{eq:Mr} below. 
\begin{lemma}\label{lem:Mr}
Given $\cH$, for all $p \in [0,1]$ and $y,r >0$ we have 
\begin{equation}\label{eq:Mr}
\Pr(M_r(\cH_p) \geq y) \le \frac{\Phi_r^{\ceil{y}}}{\ceil{y}!} \le \frac{1}{\sqrt{2\pi \ceil{y}}} \left(\frac{e \Phi_r}{\ceil{y}}\right)^{\ceil{y}} ,
\end{equation}
where $\Phi_r = \sum_{v \in V(\cH)} \Pr(|\Gamma_v(\cH_p)| \ge \ceil{r})$. 
\end{lemma}
The main idea is very intuitive: if $\cM \subseteq \cT_{r}(\cH_{p})$ attains the maximum in the definition of $M_r(\cH_p)$, 
then~$\cH_p$ contains~$|\cM|$ vertex disjoint stars $S_v=(v,W) \in \cM$, each of which `certifies' that the corresponding vertex~$v$ has degree at least~$\ceil{r}$ in~$\cH_p$ (in the sense of Section~\ref{sec:C}). Hence $M_{r}(\cH_p)=|\cM|$ events of form $\cE_v = \{|\Gamma_v(\cH_p)|\ge\ceil{r}\}$ `occur disjointly', which allows us to bring~\eqref{eq:C:ET} of Theorem~\ref{thm:CD} into play (with~$C=1$). 
\begin{proof}[Proof of Lemma~\ref{lem:Mr}]
We claim that $M_{r}(\cH_p) \le Z$ for~$Z=Z_1$ as defined in Theorem~\ref{thm:CD} with~$\cI=V(\cH)$, where~$\cE_{v}$ denotes the event that $|\Gamma_v(\cH_p)| \ge \ceil{r}$. 
This claim implies $\Pr(M_{r}(\cH_p) \ge y) \le \Pr(Z \ge y) \le \Pr(Z \ge \ceil{y})$, and we then deduce~\eqref{eq:Mr} by applying~\eqref{eq:C:ET} with $C=1$. 

To establish $M_{r}(\cH_p) \le Z$, we pick any $\cM \subseteq \cT_{r}(\cH_{p})$ which attains the maximum in~\eqref{def:Mr}, so that $M_{r}(\cH_p)=|\cM|$. 
For every $r$-star $S_v=(v,\{f_{1,v}, \ldots, f_{\ceil{r},v}\}) \in \cM$ we know that $V(S_v)=\bigcup_{1 \le i \le \ceil{r}} f_{i,v} \subseteq V_p(\cH)$ holds, which in turn implies~$\cE_v$. 
In other words, the presence of the vertices $V(S_v) \subseteq V_p(\cH)$ constitutes a certificate for the event~$\cE_v$ (using the notation of Section~\ref{sec:C}, we have $[\omega]_{V(S_v)} \subseteq \cE_{v}$). 
By definition of~$M_{r}(\cH_p)$ these certificates $\bigl(V(S_v)\bigr)_{S_v \in \cM}$ are all disjoint, so $Z \ge |\cM|=M_{r}(\cH_p)$, as claimed.  
\end{proof}
To summarize our proof framework: Lemmas~\ref{lem:Xr}--\ref{lem:Mr} apply to \emph{arbitrary} hypergraphs $\cH$ with $\max_{f \in \cH}|f| \le k$, and they basically reduce the upper tail problem for~$X=e(\cH_p)$ to the upper tail problem for the degrees of~$\cH_p$, i.e., to $\Phi_x = \sum_{v} \Pr(|\Gamma_v(\cH_p)| \ge \ceil{x})$; see also~\eqref{eq:heur:PPinduct} below. 
(These ideas are developed further in~\cite{UTlog}.) 

In general, by noting $\Pr(|\Gamma_v(\cH_p)| \ge \ceil{x}) \le \Pr(|\Gamma_v(\cH_p)| \ge \ceil{x} \mid v \in V_p(\cH))$ there is room for induction (on the number of vertices per edge), analogous to~\cite{DL,KimVu2000}. 
However, for the purposes of Theorem~\ref{thm:HGp} and~\ref{thm:HG} it seems easier to exploit the codegree condition $\Delta_2(\cH) = O(1)$ more directly (see the proof of Lemma~\ref{lem:MrH}).

\subsubsection{Sketch of the upper bound of Theorem~\ref{thm:HGp}}\label{sec:heur}
In this section we sketch the proof of upper bound of Theorem~\ref{thm:HGp}, illustrating the discussed proof framework. 
As we shall see, the desired `overshooting' phenomenon (which yields the extra $\log(e/p)$ factor in the exponent) arises naturally. 
First, using Lemma~\ref{lem:approx:basic}, for all $r,y,z>0$ satisfying $\indic{y>1}k\ceil{r}yz \le \eps\mu/2$ we obtain
\begin{equation}\label{eq:heur:PPinduct}
\Pr(X \ge (1+\eps)\mu) \le \Pr(X_r \ge (1+\eps/2)\mu) + \Pr(M_r(\cH_p) \ge y) + \indic{y > 1}\Pr(\Delta_1(\cH_p) \ge z) .
\end{equation}
(To clarify: for the indicator $\indic{y > 1}$ we exploited that $M_r(\cH_p) < 1$ implies $M_r(\cH_p) =0$, which in turn entails $\Delta_1(\cH_p) < r$.)  
Turning to further estimates of the right-hand side of~\eqref{eq:heur:PPinduct}, for $\eps=\Theta(1)$  
Lemma~\ref{lem:Xr} yields  
\begin{equation*}\label{eq:heur:Xr}
\Pr(X_r \ge (1+\eps/2)\mu) \le \exp\Bigl(-\Theta\bigl(\mu/r\bigr) \Bigr) . 
\end{equation*}
This suggests that, in order to `match' the exponent of our target bound~\eqref{eq:thm:HGp:UB}, we should pick 
\begin{equation}\label{eq:heur:r}
r = \Theta\bigl(\max\{1, \: \sqrt{\mu}/\log(e/p)\}\big) .
\end{equation}
It later turns out, see~\eqref{eq:thm:proofUBp:cond}, that this natural choice satisfies $np^{k-1}/r =o(p^{1/4})$ for $k \ge 3$ (this fails for $k=2$). 
In view of~\eqref{eq:heur:maxd}, we thus expect to obtain an extra $\log(e/p)$ factor in the exponent for $x \ge r$: 
\begin{equation}\label{eq:heur:maxd3}
\Phi_x = \sum_{v \in V(\cH)}\Pr(|\Gamma_v(\cH_p)| \ge \ceil{x}) \le \left[\left(\frac{p}{e}\right)^{1/4}\right]^{\Theta(x)} = \exp\Bigl(-\Theta\bigl(x\log(e/p)\bigr)\Bigr) .
\end{equation}
By Lemma~\ref{lem:Mr}  
it thus seems plausible 
that for $x \ge r$ we have 
\begin{equation}\label{eq:heur:mr}
\Pr(M_x(\cH_p) \geq y) \le \bigl(\Phi_x\bigr)^{\ceil{y}} \le \exp\Bigl(-\Theta\bigl(xy\log(e/p)\bigr)\Bigr) .
\end{equation}
Combining our heuristic findings with Lemma~\ref{obs:MrT}, for $\eps=\Theta(1)$ and $z \ge r$ we thus expect that  
\begin{equation}\label{eq:heur:ineq}
\Pr(X \ge (1+\eps)\mu) \le \exp\Bigl(-\Theta\bigl(\mu/r\bigr)\Bigr)  + \exp\Bigl(-\Theta\bigl(ry\log(e/p)\bigr)\Bigr) + \indic{y > 1}\exp\Bigl(-\Theta\bigl(z\log(e/p)\bigr)\Bigr) .
\end{equation}
To `match' the exponent of our target bound~\eqref{eq:thm:HGp:UB}, in view of~\eqref{eq:heur:r} 
it seems natural to set $y=z/r$ and $z=\sqrt{\eps \mu/(4k)}$, say. 
In fact, these choices also satisfy two further technical conditions used above. 
Namely, that $k\ceil{r}yz \le 2kryz = 2k z^2 \le \eps \mu/2$ holds, and that $y > 1$ implies $z \ge r$. 
Hence, if $r$ is chosen as in~\eqref{eq:heur:r}, then for $\eps=\Theta(1)$ and $\mu \ge 1$ we expect that 
\begin{equation}\label{eq:heur:final}
\Pr(X \ge (1+\eps)\mu) \le \exp\Bigl(-\Theta\bigl(\min\bigl\{\mu,\sqrt{\mu}\log(e/p)\bigr\}\bigr)\Bigr) ,
\end{equation}
which `matches' the target bound~\eqref{eq:thm:HGp:UB} of Theorem~\ref{thm:HGp}. 
With hindsight, the freedom that via $M_r(\cH_p)$ we can pick $z \gg r$ in~\eqref{eq:heur:PPinduct} seems key for going beyond the more basic decomposition~\eqref{eq:PPinduct}.

\subsubsection{Proof of the upper bound of Theorem~\ref{thm:HGp}}
In this section we follow our heuristic proof sketch, and establish the upper bound of Theorem~\ref{thm:HGp}. 
We start with the size of the largest $r$-star matching~$M_r(\cH_p)$, and make the upper tail estimate~\eqref{eq:heur:mr} rigorous via Lemma~\ref{lem:MrH} below 
(its statement is formulated with an eye on on the upcoming proof of Theorem~\ref{thm:HG}, where the $n^2 (\max\{y,1\})^{3/2} \ge 1$ term facilitates union bound arguments).   
The technical assumption~\eqref{eq:T:cond} intuitively ensures that vertices with degree at least~$r$ are sufficiently concentrated (recall that the expected degree should be $O(np^{k-1})$, see the discussion in Section~\ref{sec:log}). 
For example, $r=C(1+np^{k-1})$ satisfies~\eqref{eq:T:cond} when $np^{k-1} \ge \log n$ or $np^{k-1} \le n^{-\gamma}$ for~$C=C(\gamma,B,k,D)$ sufficiently large, but for $np^{k-1} \approx 1$ a somewhat larger choice of~$r$ seems necessary (unless we impose additional constraints on~$y$ in~\eqref{eq:MrH} below). 
By the heuristics of Section~\ref{sec:heur}, for~$r$ as defined in~\eqref{eq:heur:r} we expect that $np^{k-1}/x \le p^{1/4}$ holds in inequality~\eqref{eq:MrH}, i.e., as in~\eqref{eq:heur:mr} we should gain an extra logarithmic factor in the exponent of the upper tail by `overshooting'. 
\begin{lemma}\label{lem:MrH}
Given $k \ge 2$, $a >0$ and $D \ge 1$, 
let $\cH=\cH_n$ be a $k$-uniform hypergraph satisfying $v(\cH) \le Dn$ and $\Delta_2(\cH) \le D$. 
Then
there are $B,n_0 \ge 1$ (depending on $k,D$), such that for all $n \ge n_0$, $p \in [0,1]$, $r > 0$ satisfying 
\begin{equation}\label{eq:T:cond}
\bigl(Bnp^{k-1}/r\bigr)^{r} \le n^{-8kD} 
\end{equation}
the following holds. For all $x \ge r$ and $y > 0$ we have
\begin{equation}\label{eq:MrH}
\Pr(M_{x}(\cH_p) \ge y) \le \frac{1}{n^2 (\max\{y,1\})^{3/2}} \left(\frac{np^{k-1}}{e x}\right)^{xy/(2kD)} .
\end{equation}
\end{lemma}
Our plan is to deduce Lemma~\ref{lem:MrH} from inequality~\eqref{eq:Mr} of Lemma~\ref{lem:Mr}, and in view of the parameter~$\Phi_x = \sum_{v \in V(\cH)} \Pr(|\Gamma_v(\cH_p)| \ge \ceil{x})$ we thus study the degrees~$|\Gamma_v(\cH_p)|$. 
Here our main observation is simple. Namely, as discussed in Section~\ref{sec:log}, every edge $e \in \Gamma_v(\cH)$ intersects at most $k \Delta_2(\cH)=O(1)$ edges $f \in \Gamma_v(\cH)$, which suggests that the dependencies between the edges in $\Gamma_v(\cH_p)$ are extremely weak. 
It thus seem plausible that, conditioned on~$v \in V_p(\cH)$, the tails of $|\Gamma_v(\cH_p)|$ are comparable to those of $\Bin(|\Gamma_v(\cH)|,p^{k-1})$ with $|\Gamma_v(\cH)|p^{k-1} = O(np^{k-1})$, see also~\eqref{eq:heur:Binx}--\eqref{eq:heur:maxd}.  
This line of reasoning can easily be made rigorous via Theorem~\ref{thm:C}, 
but below we take a more direct combinatorial route (which suffices for our purposes). 
\begin{proof}[Proof of Lemma~\ref{lem:MrH}]
It suffices to prove that for all $x \ge r$ and $n \ge n_0(D)$ we have  
\begin{equation}\label{eq:Mrh:Phir}
\Phi_{x} = \sum_{v \in V(\cH)} \Pr(|\Gamma_v(\cH_p)| \ge \ceil{x}) \le \frac{1}{e n^2} \left(\frac{np^{k-1}}{ex}\right)^{x/(2kD)} .
\end{equation}
Indeed, since $y>0$ implies $\ceil{y} \ge \max\{y,1\}$, by applying~\eqref{eq:Mr} of Lemma~\ref{lem:Mr} it then follows that  
\begin{equation*}\label{eq:MrH:2}
\Pr(M_{x}(\cH_p) \ge y) \le \Pr(M_{x}(\cH_p) \ge \ceil{y}) \le \frac{\bigl(e \Phi_{x}\bigr)^{\ceil{y}}}{\sqrt{\ceil{y}} \cdot \ceil{y}^{\ceil{y}}} \le \frac{\left(\frac{np^{k-1}}{e x}\right)^{xy/(2kD)}}{n^2 (\max\{y,1\})^{3/2}} .
\end{equation*}

In the remainder we verify inequality~\eqref{eq:Mrh:Phir}, by focusing on combinatorial implications of the degree event $|\Gamma_v(\cH_p)| \ge \ceil{x}$. 
To this end we pick a subset $W \subseteq \Gamma_v(\cH_p)$ of the edges which is size maximal subject to the restriction that all edges of $W$ are \emph{vertex disjoint} outside of the centre vertex~$v$, i.e., that all distinct edges $f_i,f_j \in W$ satisfy  $(f_i \cap f_j)\setminus \{v\}=\emptyset$.
Note that for every edge $e \in \Gamma_v(\cH_p)$ there are a total of (including~$e$ itself) at most $k \Delta_2(\cH) \le kD=C$ edges $f \in \Gamma_v(\cH_p)$ with $(f \cap e)\setminus\{v\} \neq \emptyset$ (because all such edges~$f$ contain~$v$ and at least one vertex from $e \setminus \{v\}$). 
Hence, $|\Gamma_v(\cH_p)| \ge \ceil{x}$ implies 
\[
|W| \ge |\Gamma_v(\cH_p)|/C \ge x/C.
\]
Since the union of all edges in~$W$ contains exactly $|\bigcup_{f \in W} f| = 1+ (k-1)|W|$ vertices, it follows that 
\begin{equation*}\label{eq:ND1}
\Pr(|\Gamma_v(\cH_p) \ge \ceil{x}) \le \binom{|\Gamma_v(\cH)|}{\ceil{x/C}} p^{1+ (k-1)\ceil{x/C}} .
\end{equation*}
Recalling $|\Gamma_v(\cH)| \le |V(\cH)| \Delta_2(\cH) \le D^2n$, $\binom{m}{z} \le (em/z)^z$ and $p \le 1$, we obtain 
\begin{equation}\label{eq:ND}
\begin{split}
\Pr(|\Gamma_v(\cH_p)| \ge \ceil{x}) & \le 
\binom{\floor{D^2n}}{\ceil{x/C}}  p^{(k-1)\ceil{x/C}} \le 
\left(\frac{eD^2C np^{k-1}}{x}\right)^{\ceil{x/C}} .
\end{split}
\end{equation}
Defining $B=e^3D^4C^2$, using $C=kD$, $x \ge r$, and the assumption~\eqref{eq:T:cond} it follows that 
\begin{equation*}\label{eq:ND:2}
\Pr(|\Gamma_v(\cH_p)| \ge \ceil{x}) \le \left(\frac{B np^{k-1}}{r} \cdot \frac{np^{k-1}}{ex}\right)^{x/(2kD)}  \le n^{-4} \cdot \left(\frac{np^{k-1}}{ex}\right)^{x/(2kD)} .
\end{equation*}
Recalling~$|V(\cH)| \le Dn$, this readily establishes inequality~\eqref{eq:Mrh:Phir} for $n \ge n_0(D)$, completing the proof. 
\end{proof}
For the interested reader we remark that from the above proof idea it, e.g., also directly follows that 
\begin{equation*}
\begin{split}
\Pr(M_{r}(\cH_p) \ge x) & \le \sum_{\substack{U \subseteq V(\cH):\\ |U|=\ceil{x}}} \left[\prod_{v \in U} \binom{|\Gamma_v(\cH)|}{\ceil{r/C}}\right] p^{(1+ (k-1)\ceil{r/C})|U|} ,
\end{split}
\end{equation*}
which can alternatively be used to derive~\eqref{eq:MrH}. 
We find our general BK-inequality based approach more informative and flexible (e.g., with respect to possible extensions and generalizations, see~\cite{UTlog}).

We are now ready to prove the upper bound of Theorem~\ref{thm:HGp}.  
Below we shall first pick~$r$~as in~\eqref{eq:heur:r}, and then closely mimic the heuristic considerations~\eqref{eq:heur:ineq}--\eqref{eq:heur:final} of Section~\ref{sec:heur}. 
Only afterwards we verify $np^{k-1}/r = O(p^{1/4})$, the technical condition~\eqref{eq:T:cond}, and the heuristic tail inequality~\eqref{eq:heur:mr}.  
\begin{proof}[Proof of~\eqref{eq:thm:HGp:UB} of Theorem~\ref{thm:HGp}]
With foresight, we define
\begin{equation}\label{eq:HGp:sgA}
s= \log(e/p^{\gamma}), \qquad \gamma = 1/4, \qquad \text{and} \qquad  A = \max\bigl\{eB/\sqrt{a}, \: 16k^2D/\gamma\bigr\},
\end{equation}
where $B=B(k,D) \ge 1$ is as in Lemma~\ref{lem:MrH}. Furthermore, analogous to our heuristic outline, we set 
\begin{equation}\label{eq:HGp:rzy}
r=A \max\bigl\{1, \: \sqrt{\mu}/s\bigr\} , \qquad z = \sqrt{\eps\mu/(4k)},  \qquad \text{and} \qquad y = z/r ,
\end{equation}
so that $k\ceil{r}yz \le 2k z^2 = \eps \mu/2$. 
Since $y>1$ implies $z \ge r$, using inequality~\eqref{eq:heur:PPinduct} and Lemma~\ref{obs:MrT} we obtain  
\begin{equation}\label{eq:HGp:eq1}
\Pr(X \ge (1+\eps)\mu) \le \Pr(X_r \ge \mu+\eps\mu/2) + \Pr(M_r(\cH_p) \ge y) + \indic{z \ge r}\Pr(M_z(\cH_p) \ge 1) .
\end{equation}
We defer the proof of the technical claim that for all for $x \ge r$ and $y>0$ we have 
\begin{equation}\label{eq:HGp:MrH}
\Pr(M_{x}(\cH_p) \ge y) \le \exp\Bigl(-\frac{xys}{2kD}\Bigr) .
\end{equation}
Inserting~\eqref{eq:HGp:MrH} into~\eqref{eq:HGp:eq1}, using Lemma~\ref{lem:Xr}, $ry=z$ and the definitions of $r,z$ from~\eqref{eq:HGp:rzy} we infer 
\begin{equation*}\label{eq:HGp:eq2}
\begin{split}
\Pr(X \ge (1+\eps)\mu) & \le \exp\Bigl(-\frac{\min\{\eps,\eps^2\} \mu}{12kr}\Bigr) + 2\exp\Bigl(-\frac{zs}{2kD}\Bigr) \\
& = \exp\Bigl(-\frac{\min\{\eps,\eps^2\} \min\bigl\{\mu,\sqrt{\mu}s\bigr\}}{12kA}\Bigr) + 2\exp\Bigl(-\frac{\sqrt{\eps\mu}s}{2kD \sqrt{4k}}\Bigr) .
\end{split}
\end{equation*}
Noting $s\ge \gamma \log(e/p)$ and $\min\{\eps,\eps^2,\sqrt{\eps}\} = \min\{\eps^2,\eps^{1/2}\}$, there  
is $d=d(k,A,D,\gamma)>0$ such that   
\begin{equation}\label{eq:thm:proofUBp:2}
\Pr(X \ge (1+\eps)\mu) \le 3\exp\Bigl( - d \min\{\eps^2,\eps^{1/2}\} \min\{\mu,\sqrt{\mu}\log(e/p)\} \Bigr) = : 3 \exp\Bigl(-\Psi\Bigr).  
\end{equation}
We claim that \eqref{eq:thm:HGp:UB} holds with $c(\eps)=b \min\{\eps^3,\eps^{1/2}\}$ and $b=d/6$. 
In the main case $\Psi \ge 3$ this is obvious (as $3 e^{-5\Psi/6} \le 1$ and $\min\{\eps^2,\eps^{1/2}\} \ge \min\{\eps^3,\eps^{1/2}\}$). 
In the degenerate case $1 \ge \Psi/3$, Markov's inequality yields 
\[
\Pr(X \ge (1+\eps)\mu) \le \frac{1}{1+\eps} = 1-\frac{\eps}{1+\eps} \le \exp\Bigl(-\frac{\eps}{1+\eps}\Bigr) \le \exp\Bigl(-\frac{\eps\Psi}{3(1+\eps)}\Bigr),
\]
which due to $\eps/(1+\eps) \cdot \min\{\eps^2,\eps^{1/2}\} \ge \min\{\eps^3,\eps^{1/2}\}/2$ establishes the claim. 

In the remainder we verify the claimed estimate~\eqref{eq:HGp:MrH}. 
Our below proof is based on Lemma~\ref{lem:MrH}, which requires us to check the technical condition~\eqref{eq:T:cond}. 
Calculus shows that 
\begin{equation}
\label{eq:ps}
p^{\gamma}s = p^{\gamma}\log(e/p^{\gamma}) \le 1 .
\end{equation}
Using $r \ge A \sqrt{\mu}/2$, $\mu =e(\cH)p^k \ge an^2p^k$, and $k \ge 3$ (this is the only time $k \ge 2$ is not enough), we obtain  
\begin{equation}\label{eq:thm:proofUBp:cond}
\frac{np^{k-1}}{r} \le \frac{np^{k-1}s}{A \sqrt{\mu}} \le \frac{p^{(k-2)/2}s}{A \sqrt{a}} \le \frac{p^{1/2}s}{A \sqrt{a}} = \frac{p^{2\gamma}s}{A \sqrt{a}} \le \frac{p^{\gamma}}{eB} . 
\end{equation}
which also implies $r \ge eB np^{k-1-\gamma}$. 
Observe that $p \ge n^{-1/(2k)}$ implies $r \ge n^{1/2}$, say, and that $p \le n^{-1/(2k)}$ implies $p^{\gamma} \le n^{-\gamma/(2k)}$. 
Using $r \ge A$, for $n \ge n_0(k,D)$ we thus infer 
\begin{equation*}\label{eq:HGp:r}
\bigl(Bnp^{k-1}/r\bigr)^{r} \le \bigl(p^{\gamma}/e\bigr)^r \le 
\min\{e^{-r}, \: p^{\gamma A}\} \le 
\indic{p > n^{-1/(2k)}} e^{-n^{1/2}} + \indic{p \le n^{-1/(2k)}} n^{-\gamma A/(2k)} \le n^{-8kD} , 
\end{equation*}
establishing~\eqref{eq:T:cond}. 
As~\eqref{eq:thm:proofUBp:cond} and $B \ge 1$ imply $np^{k-1}/x \le e^{-s}$ for all $x \ge r$, 
inequality~\eqref{eq:MrH} of Lemma~\ref{lem:MrH} now readily establishes the technical estimate~\eqref{eq:HGp:MrH}, completing the proof.  
\end{proof}

Since our proofs are based on applications of Theorem~\ref{thm:C} and~\ref{thm:CD}, using Remark~\ref{rem:C} and~\ref{rem:CD} it is not difficult to see that all arguments carry over (essentially unchanged) 
to the \emph{uniform model} $\cH_m=\cH[V_m(\cH)]$ with $k \le m \le v(\cH)$ and $p=m/v(\cH)$, say, where $V_m(\cH) \subseteq V(\cH)$ with $|V_m(\cH)|=m$ is chosen uniformly at random (note that $e(\cH_m)=0$ if $m < k$). 
A similar remark also applies to the \emph{weighted case}, where $X=\sum_{e \in \cH_p} w_e$ for positive constants $w_e \in [\tilde{a},\tilde{D}]$, say. 
In both cases we leave the straightforward details to the interested reader (these variations also carry over to the upcoming proofs of Section~\ref{sec:CO}).

\subsection{Some refinements (proof of the upper bound of Theorem~\ref{thm:HG})}\label{sec:CO} 
In this section we refine our basic proof framework, and establish the more precise upper bound~\eqref{eq:thm:HG:UB} of Theorem~\ref{thm:HG}. 
Recall that the exponent of~\eqref{eq:thm:HG:UB} is essentially either of sub-Gaussian type $\exp\bigl(-ct^2/\Var X\bigr)$ or clustered type $\exp\bigl(-c\sqrt{t}\log(1/p)\bigr)$; see also~\eqref{eq:thm:HG}. 
Heuristically speaking, the corresponding phase transition near $(\Var X)^{2/3}$ causes some technical difficulties for the approach taken in Section~\ref{sec:CA} (for $p \ge n^{-1/(k-1)+o(1)}$ it turns out that sharp tail estimates are easier when~$t$ is far away from~$(\Var X)^{2/3}$).  
Here one bottleneck is Lemma~\ref{lem:approx:basic}, which on an intuitive level only distinguishes between two ranges of the degrees: smaller and larger than~$r$. 
In this section we shall rectify this issue, by distinguishing between a wide range of different degrees. 

More concretely, our refined sparsification strategy is to \emph{iteratively} decrease the maximum degree of~$\cH_p$, 
until we are able to bound the number of remaining edges by $X_r$ as defined in~\eqref{def:Xr}. 
Using the convention $\NN = \{0,1, \ldots\}$, we shall eventually 
implement this strategy via $\cT(\beta,\gamma,r,t)$, which is the event that 
\begin{alignat}{2}
\label{eq:Njs}
M_{r_j}(\cH_p) &< \beta \sqrt{t}s/r_j &\qquad&\text{for all $j \in \NN$ with $r_j < \sqrt{t}/s$, and} \\
\label{eq:Nj}
M_{r_j}(\cH_p) &< \beta \sqrt{t}/r_j &&\text{for all $j \in \NN$ with $r_j \ge \sqrt{t}/s$,} 
\end{alignat}
where we tacitly used the following convenient parametrization:
\begin{equation}\label{eq:par}
\begin{split}
s = s(\gamma) &= \log(e/p^{\gamma}),\\ 
r_j =r_j(r) &=2^j r.
\end{split}
\end{equation}
(The intricate form of~\eqref{eq:Njs}--\eqref{eq:Nj} is hard to digest on first sight; both events are based on 
a delicate interplay 
between the combinatorial and probabilistic estimates in the upcoming proofs of Lemma~\ref{lem:approx} and~\ref{lem:T}.)

The following combinatorial lemma intuitively states that $X \approx X_r$ whenever $\cT(\beta,\gamma,r,t)$ holds. 
\begin{lemma}\label{lem:approx}
Given $k \ge 1$, suppose that $\cH$ satisfies $\max_{f \in \cH} |f| \le k$. Then, for all $\beta \in \big(0,1/(32k)]$, $r \ge 1$ and $\gamma,t>0$, the event  $\cT(\beta,\gamma,r,t)$ implies $X_r \le X \le X_r + t/2$. 
\end{lemma}
The idea is to iterate the proof of Lemma~\ref{lem:approx:basic}: using the resulting hypergraph sequence~$\cH_p=\cG_J \supseteq \cdots \supseteq \cG_0$ we shall 
estimate $X=e(\cH_p)$ in terms of the step-wise differences: $X=e(\cG_0) + \sum_{0 \le j < J}[e(\cG_{j+1})-e(\cG_{j})]$.  
The definition of $\cT(\beta,\gamma,r,t)$ then ensures that $\sum_{0 \le j < J}[e(\cG_{j+1})-e(\cG_{j})] \le t/2$ and $e(\cG_0) \le X_r$ hold. 
\begin{proof}[Proof of Lemma~\ref{lem:approx}]
The lower bound $X =e(\cH_p) \ge X_r$ is trivial, so we henceforth focus on the upper bound. Let $J$ be the smallest integer $J \ge 0$ with $r_J \ge \sqrt{t}$. 
We now construct the sequence $(\cG_j)_{0 \le j \le J}$ with $\cG_J=\cH_p$ and  $\Delta_1(\cG_j) \le \floor{r_j}$. 
For $\cG_J=\cH_p$, observe that~\eqref{eq:Nj} and $\beta \le 1 \le s$ imply $M_{r_j}(\cH_p) < \beta \le 1$ for all $r_j \ge \sqrt{t}$. 
Hence, since $\Delta_1(\cH_p) \ge \ceil{r_j}$ implies $M_{r_j}(\cH_p) \ge 1$, it follows that $\Delta_1(\cG_J) = \Delta_1(\cH_p) \le \ceil{r_J}-1 \le \floor{r_J}$. 
Given $\cG_{j+1}$ with $0 \le j < J$, we fix some $\cM \subseteq \cT_{\ceil{r_j}}(\cG_{j+1})$ which attains the maximum in \eqref{def:Mr}, so that $|\cM|=M_{r_j}(\cG_{j+1}) \le  M_{r_j}(\cG_{J}) = M_{r_j}(\cH_p)$ by monotonicity. 
We remove all edges from $\cG_{j+1}$ which contain at least one vertex from some $r_j$-star $S \in \cM$, and denote the resulting hypergraph by~$\cG_j$. 
Hence $\Delta_1(\cG_j) \le \ceil{r_j}-1 \le \floor{r_j}$, because otherwise we could add another $r_j$-star to $\cM$ (contradicting the maximality of~$|\cM|$). 

Next we estimate $X=e(\cH_p)$ in terms of the hypergraph sequence $(\cG_j)_{0 \le j \le J}$. Since each $r_j$-star consists of $\ceil{r_j}$ edges, for $0 \le j < J$ it follows by construction and monotonicity (using $M_{r_j}(\cG_{j+1}) \le  M_{r_j}(\cH_p)$, $\ceil{r_j} \le r_j+1 \le 2 r_j$ and $\Delta_1(\cG_{j+1}) \le r_{j+1} = 2 r_j$) that 
\[
e(\cG_{j+1})-e(\cG_{j}) \le M_{r_j}(\cG_{j+1}) \cdot \ceil{r_j} \max_{f \in \cH} |f| \cdot \Delta_1(\cG_{j+1}) \le M_{r_j}(\cH_{p}) \cdot 4k r_j^2. 
\]
Hence, using $\cH_p=\cG_J$, \eqref{eq:Njs}--\eqref{eq:Nj} and $\max_{0 \le j < J} r_j \le \sqrt{t}$ we readily obtain 
\[
X = e(\cG_J) \le e(\cG_0) + 4k\sum_{0 \le j < J} M_{r_j}(\cH_{p}) r_j^2 \le e(\cG_0) + 4\beta k \sqrt{t}\Bigl(s\sum_{\substack{0 \le j < J:\\ r_j \le \sqrt{t}/s}}r_j + \sum_{\substack{0 \le j < J:\\ \sqrt{t}/s \le r_j \le \sqrt{t}}}r_j\Bigr) .
\]
For any $z > 0$, in view of $r_{j}=2^jr$ it is easy to see that 
\begin{equation}\label{eq:sum}
\sum_{j \in \NN: r_j \le z} r_j = z \sum_{j \in \NN: r_j \le z} r_j/z \le z \sum_{j \in \NN} 2^{-j} = 2z .
\end{equation}
Thus, noting that $\Delta_1(\cG_0) \le \floor{r_0} \le r$ implies $e(\cG_0) \le X_r$, using $\beta \le 1/(32k)$ it follows that 
\[
X \le e(\cG_0) + 16\beta k t \le X_r + t/2 ,
\]
completing the proof.  
\end{proof} 
In view of Lemma~\ref{lem:Xr} and~\ref{lem:approx}, we now focus on the probability of the event~$\neg\cT(\beta,\gamma,r,t)$. 
Ignoring some technical assumptions (which are similar to those of Lemma~\ref{lem:MrH}), the following result essentially states that $\Pr(\neg\cT(\beta,\gamma,r,t))$ is negligible for our purposes 
(the $1/n$ prefactor in~\eqref{eq:T} is ad-hoc, and eventually becomes the usually irrelevant $n^{-1}$ term in~\eqref{eq:thm:HG:UB} of Theorem~\ref{thm:HG}). 
\begin{lemma}\label{lem:T}
Given $k \ge 3$, $a >0$ and $D \ge 1$, 
let $\cH=\cH_n$ be a $k$-uniform hypergraph satisfying $v(\cH) \le Dn$, $e(\cH) \ge an^2$ and $\Delta_2(\cH) \le D$. 
Set $X=e(\cH_p)$, $\mu = \E X$ and $\varphi(x)=(1+x)\log(1+x)-x$. 
Then there are $B,n_0 \ge 1$ (depending on $k,D$), such that for all $n \ge n_0$, $p \in (0,1]$, $\beta \in (0,1]$, $\gamma \in (0,1/8]$, and $r,t>0$ satisfying~\eqref{eq:T:cond}
we have
\begin{equation}\label{eq:T}
\Pr(\neg\cT(\beta,\gamma,r,t)) 
\le \frac{1}{n} \exp\left(-\frac{\min\{a,\beta\}}{2kD} \min\left\{\frac{\varphi(t/\mu)\mu^2}{\Lambda},\sqrt{t}s\right\}\right) .
\end{equation}
\end{lemma}
The definition of $\cT(\beta,\gamma,r,t)$ is, in some sense, already a significant part of the proof. 
Indeed, writing $C=2kD$, our argument hinges on the fact that~\eqref{eq:MrH} of Lemma~\ref{lem:MrH} yields, in our case, a bound of the form  
\[
\Pr(M_{r_j}(\cH_p) \ge y) \le \frac{1}{n^2}\min\left\{e^{-r_j y/C}, \: \left(\frac{np^{k-1}}{er_j}\right)^{r_j y/C}\right\} .
\] 
Hence  
$\Pr(M_{r_j}(\cH_p) \ge \beta\sqrt{t}s/r_j) \le n^{-2}e^{-\beta\sqrt{t}s/C}$. 
Furthermore, for $r_j \ge \sqrt{t}/s$ it turns out that usually $np^{k-1}/(er_j) \le p^{\gamma}/e = e^{-s}$ holds, 
so $\Pr(M_{r_j}(\cH_p) \ge \beta\sqrt{t}/r_j) \le n^{-2} e^{-\beta\sqrt{t}s/C}$ by `overshooting'. 
Recalling~\eqref{eq:Njs}--\eqref{eq:Nj}, using a careful union bound argument this reasoning eventually establishes inequality~\eqref{eq:T}.  
\begin{proof}[Proof of Lemma~\ref{lem:T}]
Let $C=2kD$. We use $B=B(k,D) \ge 1$ as given by Lemma~\ref{lem:MrH}, so that~\eqref{eq:MrH} holds for all $x=r_j$ and $y>0$. 
Note that~\eqref{eq:T:cond} entails $r \ge Bnp^{k-1} \ge np^{k-1}$.  
With \eqref{eq:MrH} in hand, we now estimate $\Pr(\neg\cT(\beta,\gamma,r,t))$ by a delicate union bound argument. 
With foresight, we first assume~$r \ge a\Phi$, where 
\begin{equation}\label{eq:T:Psi}
\Phi = \frac{\varphi(t/\mu)\mu}{np^{k-1}} .
\end{equation}
Note that $M_{r_0}(\cH_p)=0$ entails $M_{r_j}(\cH_p) =0$ for all $j \ge 0$, which in view of \eqref{eq:Njs} and \eqref{eq:Nj} implies $\cT(\beta,\gamma,r,t)$.  
Hence, using $r_0 = r \ge \max\{np^{k-1}, a\Phi\}$ and~\eqref{eq:MrH}, we infer 
\begin{equation}\label{eq:T:1}
\Pr(\neg\cT(\beta,\gamma,r,t)) \le \Pr(M_{r_0}(\cH_p) > 0) = \Pr(M_{r}(\cH_p) \ge 1) \le 
\frac{1}{n} \left(\frac{np^{k-1}}{e r}\right)^{r/C} \le \frac{1}{n} \exp\Bigl(-a\Phi/C\Bigr) . 
\end{equation}
We henceforth assume~$r < a\Phi$. 
Using Lemma~\ref{lem:MrH}, $r_j=2^jr \ge np^{k-1}$ and $s \ge 1$, we infer for $n \ge n_0(\beta)$ that 
\begin{equation}\label{eq:T:sum1}
\begin{split}
\Pr(\text{\eqref{eq:Njs} fails}) & \le \sum_{j \in \NN: r_j \le \sqrt{t}/s} \Pr(M_{r_j}(\cH_p) \ge \ceil{\beta\sqrt{t}s/r_j}) \\
& \le \sum_{j \in \NN: r_j \le \sqrt{t}/s} \frac{r_j^{3/2}}{n^2(\beta\sqrt{t}s)^{3/2}} \cdot \exp\Bigl(-\beta\sqrt{t}s/C\Bigr) \le \frac{1}{2n} \exp\Bigl(-\beta\sqrt{t}s/C\Bigr),
\end{split}
\end{equation}
where the last inequality follows analogously to~\eqref{eq:sum}. 
Observing that $M_{r_{j+1}}(\cH_p) \ge 1$ implies $M_{r_{j}}(\cH_p) \ge 1$, for $n \ge n_0(\beta)$ a similar argument (exploiting that $r_j \ge \sqrt{t}$ implies $\beta \sqrt{t}/r_j \le \beta \le 1$) yields
\begin{equation}\label{eq:T:sum2}
\begin{split}
\Pr(\text{\eqref{eq:Nj} fails}) & \le \sum_{j \in \NN: \sqrt{t}/s \le r_j \le \max\{2\sqrt{t},r\}} \Pr(M_{r_j}(\cH_p) \ge \ceil{\beta\sqrt{t}/r_j}) \\
& \le \frac{1}{2n} \max_{j \in \NN: r_j \ge \sqrt{t}/s} \left(\frac{np^{k-1}}{e r_j}\right)^{\beta \sqrt{t}/C} \le \frac{1}{2n} \left(\frac{np^{k-1}s}{e \sqrt{t}}\right)^{\beta \sqrt{t}/C} .
\end{split}
\end{equation}
(To clarify: the condition $r_j \le \max\{2\sqrt{t},r\}$ ensures that the considered range of~$r_j$ is non-empty.) 
In the following we exploit the assumption $r < a\Phi$ to further estimate~\eqref{eq:T:sum2}. 
Note that $\log(1+x)\le x$ implies
\begin{gather}
\label{eq:varphi:x3}
\varphi(x) = (1+x)\log(1+x)-x \le x^2 .
\end{gather}
In view of~\eqref{eq:T:Psi} and~\eqref{eq:varphi:x3}, using $\Phi > r/a \ge np^{k-1}/a$ and $\mu = e(\cH)p^k \ge an^2p^k$ we deduce 
\begin{equation}\label{eq:T:t}
t^2 \ge \varphi(t/\mu)\mu^2 = \Phi \mu np^{k-1} \ge n^4 p^{3k-2}.
\end{equation}
Since $k \ge 3$ and $\gamma \le  1/8$ (in fact, $\gamma \le (k-2)/8$ suffices), using \eqref{eq:T:t} and~\eqref{eq:ps} we obtain 
\begin{equation}\label{eq:T:max2}
\frac{np^{k-1}s}{e \sqrt{t}} \le \frac{p^{(k-2)/4}s}{e} \le \frac{p^{(k-2)/4-\gamma}}{e} \le \frac{p^{1/4-\gamma}}{e} \le \frac{p^{\gamma}}{e} = e^{-s} .
\end{equation}
Now, inserting \eqref{eq:T:max2} into \eqref{eq:T:sum2}, in view of~\eqref{eq:T:sum1} we infer (for $r < a\Phi$) that 
\begin{equation*}\label{eq:T:2}
\Pr(\neg\cT(\beta,\gamma,r,t)) = \Pr(\text{\eqref{eq:Njs} or \eqref{eq:Nj} fails}) \le \frac{1}{n} \exp\Bigl(-\beta\sqrt{t}s/C\Bigr) ,
\end{equation*}
which together with \eqref{eq:T:1}, $C=kD$ and $\Phi\ge \varphi(t/\mu)\mu^2/\Lambda$ completes the proof of~\eqref{eq:T}. 
\end{proof}

We are now ready to prove the upper bound of Theorem~\ref{thm:HG}, and 
our main remaining task is to pick a suitable parameter~$r$. 
Here the technical condition~\eqref{eq:T:cond} prevents the natural choice $r=C\Lambda/\mu = \Theta(1+np^{k-1})$ when $np^{k-1} \approx 1$, 
which explains the more involved form of~$r$ in the next proof 
(this complication is only needed in the pedestrian case~(iii) below). 
\begin{proof}[Proof of~\eqref{eq:thm:HG:UB} of Theorem~\ref{thm:HG}] 
It suffices to consider the following three cases: (i)~$p \ge \gamma n^{-1/(k-1)}(\log n)^{1/(k-1)}$, (ii)~$p \le n^{-1/(k-1)-\gamma}$, and (iii)~$t \ge \min\bigl\{\gamma\min\{(\Var X)^{2/3},\mu^{2/3}\}(\log n)^{4/3},\mu p^{(k-2)/3-\gamma}\bigr\}$. 
Of course, in all cases we may assume $\gamma \le 1/8$ (decreasing $\gamma$ yields less restrictive assumptions), and in case~(iii) we may also assume $n^{-1/(k-1)-\gamma} \le p \le n^{-1/(2k)}$, say (otherwise case~(i) or~(ii) applies). 
We start by introducing several parameters. 
By Remark~\ref{rem:Var} there is a constant $b=b(k,a,D) \in (0,1]$ such that for all $p \in [0,1/2]$ we have 
\begin{equation}\label{eq:Var:b}
\Var X \ge b \Lambda .
\end{equation}
Let $\beta=1/(32 k)$. Define $s=s(\gamma)$ as in \eqref{eq:par}, and set 
\begin{equation*}\label{eq:HG:rdef}
r=A\tr, \quad 
A = \max\biggl\{\frac{3 B}{\min\{1,a^{1/2},b\}}, \: \frac{32k^2D}{\gamma^{k-1}}, \: \frac{24kD}{{\min\{1,a^{1/2},b\}}\gamma^{3/2}} \biggr\},  
 \quad \text{and} \quad  \tr=\max\biggl\{\frac{\Lambda}{\mu}, \: \frac{\varphi(t/\mu)\mu}{\sqrt{t}s}\biggr\}  ,
\end{equation*}
where $B=B(k,D)$ is as in Lemma~\ref{lem:T}. 
We defer the proof of the claim that $r$ satisfies the technical condition~\eqref{eq:T:cond}, and first apply Lemmas~\ref{lem:Xr} and~\ref{lem:approx}--\ref{lem:T}.
So, using the definition of~$r$, it follows that  
\begin{equation*}
\begin{split}
\Pr(X \ge \mu+t) & \le \Pr(X_r \ge \mu + t/2) + \Pr(\neg\cT(\beta,\gamma,r,t))\\
& \le \exp\left(-\frac{\varphi(t/\mu) \mu}{4kr} \right) + \frac{1}{n} \exp\left(-\frac{\min\{a,\beta\}}{2kD} \min\left\{\frac{\varphi(t/\mu)\mu^2}{\Lambda},\sqrt{t}s\right\}\right) \\
& \le (1+n^{-1}) \exp\left(-\frac{\min\{a,\beta,1\}}{4kA} \min\left\{\frac{\varphi(t/\mu)\mu^2}{\Lambda},\sqrt{t}s\right\}\right) .
\end{split}
\end{equation*}
Since $s=\log(e/p^{\gamma}) \ge \gamma\log(e/p)$, this establishes~\eqref{eq:thm:HG:UB} with $c=\gamma\min\{a,\beta,1\}/(4kA)$. 

In the remainder we verify the technical condition~\eqref{eq:T:cond}. For later reference, note that 
\begin{equation}\label{eq:HG:tr:lb}
\tr \ge \Lambda/\mu \ge \max\{np^{k-1}, 1\} .
\end{equation}
Recalling $r = A \tr$, in case~(i) we have $r \ge Anp^{k-1} \ge A\gamma^{k-1} \log n$, and in case~(ii) we have $np^{k-1} \le n^{-(k-1)\gamma}$ and $r  \ge A$. 
In both cases, using $r \ge \max\{eBnp^{k-1},B\}$ and $r \ge A \ge 8kD/\gamma^{k-1}$ we infer that 
\begin{equation}\label{eq:HG:r1r2}
\bigl(Bnp^{k-1}/r\bigr)^{r} \le \min\bigl\{e^{-r}, \: (np^{k-1})^r\bigr\} \le \max\bigl\{n^{-A \gamma^{k-1}}, \: n^{-A (k-1)\gamma}\bigr\} \le n^{-8kD} . 
\end{equation}
The remaining case~(iii) requires somewhat tedious case distinctions. Recalling~\eqref{eq:varphi:x4}, it follows that 
\begin{equation}\label{eq:tr:bounds}
\tr \ge \frac{\varphi(t/\mu)\mu}{\sqrt{t}s} \ge \frac{\min\{t^{1/2},t^{3/2}/\mu\}}{3s} \ge \indic{t \ge \mu}\frac{\mu^{1/2}}{3s} + \indic{t < \mu}\frac{t^{3/2}}{3\mu s}  .
\end{equation}
With foresight, note that \eqref{eq:ps} and $p \ge n^{-1/(k-1)-\gamma}$ imply, for $n \ge n_0$, that 
\begin{equation}\label{eq:s:bounds}
s=\log(e/p^{\gamma}) \le \min\{1+\gamma\log(1/p), \: p^{-\gamma}\} \le \min\{\log n, \: p^{-\gamma}\} .
\end{equation}
Using~\eqref{eq:Var:b} and $p = o(1)$ we have $\Var X \ge b \mu$, where $b \in (0,1]$.  
Combining this estimate with the  assumed lower bound for~$t$ in the case~(iii), using $\mu =e(\cH)p^{k} \ge an^2p^k$ and~\eqref{eq:s:bounds} it follows that 
\begin{equation}\label{eq:tms}
\frac{t^{3/2}}{\mu s} \ge \min\left\{\frac{\gamma^{3/2} b (\log n)^2}{s}, \:  \frac{\mu^{1/2}p^{(k-2)/2-3\gamma/2}}{s}\right\} \ge \min\left\{\gamma^{3/2}b\log n, \: a^{1/2}np^{k-1-\gamma/2}\right\}.
\end{equation}
Since $k \ge 3$ and $\gamma \le 1/8$ imply $1 \ge p^{(k-2)/2-3\gamma/2}$, note that the final expression in~\eqref{eq:tms} is also a lower bound for~$\mu^{1/2}/s$. 
In view of~\eqref{eq:tr:bounds}, we thus infer 
\begin{equation}\label{eq:HG:tr:lb2}
\tr \ge 3^{-1}\min\{a^{1/2},b\} \cdot \min\bigl\{\gamma^{3/2}\log n, \: np^{k-1-\gamma/2}\bigr\} .
\end{equation}  
If the minimum in~\eqref{eq:HG:tr:lb2} is attained by the $\gamma^{3/2}\log n$ term, then $r = A \tr \ge e B \tr$ and~\eqref{eq:HG:tr:lb} imply $(Bnp^{k-1}/r)^r \le e^{-r}=e^{-A \tr}$, 
so that $A \tr \ge 8kD \log n$ establishes~\eqref{eq:T:cond}. 
Otherwise the  minimum in~\eqref{eq:HG:tr:lb2} is attained by the $np^{k-1-\gamma/2}$ term, in which case $r=A \tr$ implies $(Bnp^{k-1}/r)^r \le (p^{\gamma/2})^r$ by choice of~$A$. 
Using $p \le n^{-1/(2k)}$ and $r=A \tr \ge A \ge 32 k^2 D/\gamma$, this readily establishes~\eqref{eq:T:cond}, completing the proof.  
\end{proof}

\section{Lower bounds}\label{sec:LB}
In this section we establish the lower bounds \eqref{eq:thm:HGp:LB} and \eqref{eq:thm:HG:LB} of Theorem~\ref{thm:HGp} and~\ref{thm:HG}. 
The proofs are based on three different `configurations' of the vertices in $V_p(\cH)$, which each yield a distinct lower bound for the upper tail of $X=e(\cH_p)$. 
The heuristic idea is that 
one of them should hopefully always approximate the most likely way to obtain $X \approx (1+\eps)\mu$ or $X \approx \mu+t$, respectively. 
In brief, we shall use configurations where many edges cluster on few vertices (Section~\ref{sec:LB:constr}), where many edges arise disjointly (Section~\ref{sec:LB:disj}), or where there are overall too many vertices (Section~\ref{sec:LB:vx}). 
Here one main novelty is on a conceptual level: in contrast to previous work we obtain, in a wide range, the correct dependence on $t=\eps \mu$. 

\subsection{Configurations with clustering}\label{sec:LB:constr}
The first lower bound is based on property $\fX(\cH,D,x)$ defined in \eqref{P}, which intuitively states that many edges can cluster on comparatively few vertices. 
In other words, enforcing $W \subseteq V_p(\cH)$ for a reasonably small set of vertices $W$ is enough to guarantee that the number of induced edges $X=e(\cH_p)=e(\cH[V_p(\cH)])$ is fairly large. 
A related approach was taken in~\cite{UTAP} and~\cite{UTSG} for arithmetic progressions and subgraphs, respectively.  
\begin{theorem}\label{thm:LB:constrp} 
Given a hypergraph $\cH$, set $X=e(\cH_p)$ and $\mu=\E X$. 
For all $D \ge 1$, $p \in (0,1]$ and $t \ge 0$ satisfying $\fX(\cH,D,\mu+t)$ and $\mu+t \ge 1$ we have 
\begin{equation}\label{eq:thm:LB:constrp}
\Pr(X \ge \mu+t) \ge \exp\Bigl(-D \sqrt{\mu+t} \log(1/p)\Bigr) .
\end{equation}
\end{theorem}
\begin{proof}
By $\fX(\cH,D,\mu+t)$ there is $W \subseteq V(\cH)$ satisfying $|W| \le D \sqrt{\mu+t}$ and $e(\cH[W]) \ge \mu+t$. 
Hence 
\[
\Pr(X \ge \mu+t) \ge \Pr(W \subseteq V_p(\cH)) = p^{|W|} \ge p^{D\sqrt{\mu+t}} ,
\]
completing the proof. 
\end{proof}
Using a new `local' variant of the above argument we now improve the $\sqrt{\mu+t}$ in the exponent of~\eqref{eq:thm:LB:constrp} to~$\sqrt{t}$, which is crucial when~$t = o(\mu)$. 
The basic idea is to `create' at least $\mu+t$ edges as follows: (i)~first we use the above clustering construction to `locally' enforce, say, $2t$ edges, and (ii)~then we use correlation inequalities and a one-sided version of Chebyshev's inequality to show that typically at least $\mu-t$ of the \emph{remaining} $r=e(\cH)-2t$ edges are present in~$\cH_p$. (The crux is that the \emph{expected} number of remaining edges is at least $rp^k=\mu-2tp^k$.) 
This approach seems of independent interest, and a similar reasoning can, e.g., be used to refine the lower bounds for subgraph counts obtained by Janson, Oleszkiewicz and Ruci{\'n}ski~\cite{UTSG}. 
\begin{theorem}\label{thm:LB:constr} 
Given $k \ge 2$, $a>0$ and $D \ge 1$, 
let $\cH=\cH_n$ be a $k$-uniform hypergraph satisfying $v(\cH) \le Dn$, $e(\cH) \ge an^2$ and $\Delta_2(\cH) \le D$. 
Set $X=e(\cH_p)$, $\mu=\E X$ and $\Lambda = \mu (1+np^{k-1})$. 
Given $\alpha \in (0,1)$, there are $n_0>0$ (depending only on $k,a,D$) and $c,\lambda \ge 1$ (depending only on $\alpha,k,a,D$) such that for all $n \ge n_0$, $p \in (0,1-\alpha]$ and $t \ge \indic{\mu \ge 1/2}\min\{\sqrt{\Var X},\sqrt{\Lambda}\}$ satisfying $\fX(\cH,D,\min\{\lambda t,\mu+t\})$ and $\mu+t \ge 1$ we have 
\begin{equation}\label{eq:thm:LB:constr}
\Pr(X \ge \mu+t) \ge \exp\Bigl(-c \sqrt{t}\log(1/p)\Bigr) .
\end{equation}
\end{theorem}
We remark that the form of the somewhat strange-looking assumption $t \ge \indic{\mu \ge 1/2}\min\{\sqrt{\Var X},\sqrt{\Lambda}\}$ will be convenient later on. 
Before giving the proof of Theorem~\ref{thm:LB:constr}, let us informally discuss the structure of the argument. 
The clustering construction intuitively `marks' a set of~$2t$ edges in~$\cH$. 
Let~$Z$ denote the number of `unmarked' edges that occur in $\cH_p$, so $\E Z = (e(\cH)-2t)p^k=\mu-2tp^k$.
The punchline is that the clustering construction (which enforces the $2t$ `marked' edges) allows us to 
shift our focus from the unlikely event $X \ge \E X +t$ to the `typical' event $Z \ge \E Z -t/2$.
Indeed, it turns out that, using Harris' inequality~\cite{Harris1960} and $\mu = \E Z + 2tp^k$, for suitable $W \subseteq V(\cH)$ with $|W| = O(\sqrt{t})$ and $e(\cH[W]) \ge 2t$ we eventually arrive at 
\begin{equation*}
\Pr(X \ge \mu+t) \; \ge \; \Pr(W \subseteq V_p(\cH)) \cdot \Pr(Z \ge \mu-t) \; \ge \; p^{\Theta(\sqrt{t})} \cdot \Pr(Z \ge \E Z -t+2tp^k) .
\end{equation*}
It seems plausible that $\Var Z = O(\Var X)$ holds. 
A folklore variant of the Paley--Zygmund inequality states that, given any random variable $Y \ge 0$, for all $0 \le t < \E Y$ we have 
\begin{equation}\label{eq:PZ}
\Pr(Y \ge \E Y - t) \ge \frac{t^2}{\Var Y + t^2} .
\end{equation}
So, assuming $p \le 1/2$ (which implies $2p^k \le 1/2$ for $k \ge 2$), for $t \ge \sqrt{\Var X}$ 
we should intuitively obtain
\[
\Pr(Z \ge \E Z -t+2tp^k) \; \ge \; \Pr(Z \ge \E Z -t/2) \; \ge \; \Omega\left(\frac{t^2}{\Var Z+t^2}\right)  = \Omega(1) . 
\]
The proof below makes this reasoning rigorous, but there are a number of subtle issues (which make the details somewhat cumbersome).  
For example, the parameter~$t$ may be very small, so we can not, as usual, ignore rounding issues. 
Furthermore, to allow for $p \le 1-\alpha$ we need to plant~$\lambda t$ copies (instead of just~$2t$ copies) for carefully chosen~$\lambda=\lambda(\alpha,k)>0$. 
In addition, the $W \subseteq V_p(\cH)$ based construction does not work if $\lambda t$ is larger than the total number of edges~$e(\cH)$, so we shall only enforce $\min\{\lambda t,\mu+t\}$ copies. 
\begin{proof}[Proof of Theorem~\ref{thm:LB:constr}]
We defer the elementary proof of the fact that there is $\lambda=\lambda(\alpha,k) > 0$ satisfying 
\begin{equation}\label{eq:LB:constr:lt}
\lambda t  \ge 2.
\end{equation}
Defining $x =\min\{\lambda t,\mu+t\}$, by $\fX(\cH,D,x)$ there is $W \subseteq V(\cH)$ satisfying $|W| \le D \sqrt{\lambda t}$ and $e(\cH[W]) \ge x$. 
To later avoid rounding issues, we pick $\beta+1 \in [\lambda/2,\lambda]$ such that $(\beta+1) t$ is an integer. 
Defining $y = \min\{(\beta+1) t,\mu+t\}$, note that there is $\cG \subseteq \cH[W]$ with $e(\cG)=\ceil{y}$.  
Define $Y=e(\cG[V_{p}(\cH)])$. Clearly, 
\begin{equation}\label{eq:thm:LB:constr:1}
\Pr(Y \ge y) = \Pr(Y \ge \min\{(\beta+1) t,\mu+t\}) \ge \Pr(W \subseteq V_p(\cH)) = p^{|W|} \ge p^{D\sqrt{\lambda t}} .
\end{equation}
In the case $\mu \le \beta t$ we have $\mu+t \le y$, so that $\Pr(X \ge \mu+t) \ge \Pr(X \ge y) \ge \Pr(Y \ge y)$ and~\eqref{eq:thm:LB:constr:1} establish inequality~\eqref{eq:thm:LB:constr} for any constant~$c$ satisfying $c \ge D\sqrt{\lambda}$ (we defer the precise choice of~$c$). 

Henceforth we focus on the more interesting case $\mu > \beta t$. Define $Z=X-Y$. 
Since $Y \ge (\beta+1) t$ and $Z \ge \mu-\beta t$ are both increasing events, using $X = Y+Z$, Harris' inequality~\cite{Harris1960}, and~\eqref{eq:thm:LB:constr:1} we infer
\begin{equation}\label{eq:thm:LB:constr:2}
\begin{split}
\Pr(X \ge \mu+t) & \ge \Pr(\text{$Y \ge (\beta+1) t$ and $Z \ge \mu-\beta t$}) \\
& \ge \Pr(Y \ge (\beta+1) t) \Pr(Z \ge \mu-\beta t) \ge p^{D\sqrt{\lambda t}} \Pr(Z \ge \mu-\beta t). 
\end{split}
\end{equation}
We defer the proof of the conceptually straightforward (but slightly tedious) claim that 
\begin{align}
\label{eq:thm:LB:EZ}
\E Y & \le (\beta - 1) t, \\
\label{eq:thm:LB:VZ}
\Var Z & \le C t^2,
\end{align}
where $C=C(k,a,D,\lambda) \ge 1$. 
Using $\E Z-t = \E X-\E Y -t \ge \mu -\beta t$ and the Paley--Zygmund inequality~\eqref{eq:PZ}, 
for $d=\log_{1-\alpha}(1/(C+1))>0$ it follows (exploiting $1-\alpha \ge p$ and $1 \le \lambda t$) that 
\begin{equation}\label{eq:thm:LB:constr:3}
\Pr(Z \ge \mu-\beta t) \ge \Pr(Z \ge \E Z-t) \ge \frac{t^2}{\Var Z +t^2} \ge \frac{1}{C+1} =(1-\alpha)^{d} \ge p^{d} \ge p^{d\sqrt{\lambda t}} . 
\end{equation}
Inserting~\eqref{eq:thm:LB:constr:3} into~\eqref{eq:thm:LB:constr:2} establishes inequality~\eqref{eq:thm:LB:constr} with $c = D\sqrt{\lambda} + d\sqrt{\lambda}$.

It remains to prove the auxiliary claims~\eqref{eq:LB:constr:lt} and~\eqref{eq:thm:LB:EZ}--\eqref{eq:thm:LB:VZ}. 
Let $\lambda = 4/(1-(1-\alpha)^k)$. 
Writing $Y_{e} = \indic{e \subseteq V_{p}(\cH)}$, note that Harris' inequality yields $\E(Y_eY_f) \ge \E Y_e \E Y_f$. 
As $\E Y_e^2 = \E Y_e$, we infer 
\begin{equation}\label{eq:thm:LB:constr:VarX}
\Var X = \sum_{(e,f) \in \cH \times \cH} \bigl[\E (Y_e Y_f) - \E Y_e \E Y_f \bigr] \ge \sum_{e \in \cH}(1-\E Y_e) \E Y_e \ge (1-p^k)\mu \ge (1-(1-\alpha)^k)\mu = 4\mu/\lambda . 
\end{equation}
Observing $\Lambda \ge \mu$ and $t \ge 1-\mu$, using the assumed lower bound for~$t$ (and $\lambda \ge 4$) it follows that   
\[
\lambda t \ge \lambda\Bigl(\indic{\mu \ge 1/2} \min\bigl\{\sqrt{4\mu/\lambda},\sqrt{\mu}\bigr\} + \indic{\mu <1/2}(1-\mu)\Bigr) \ge 2, 
\] 
establishing the claimed inequality~\eqref{eq:LB:constr:lt}. 
Recall that we only need to prove~\eqref{eq:thm:LB:EZ}--\eqref{eq:thm:LB:VZ} whenever $\mu > \beta t$. 
In this case $\ceil{y}=(\beta+1)t$ holds by choice of~$\beta$, so that $(1-\alpha)^k = 1-4/\lambda$ and $\beta+1 \ge \lambda/2$ imply 
\[
\E Y = \ceil{y} p^k \le (\beta+1) (1-\alpha)^kt  = \bigl[\beta+1 - (\beta+1)4/\lambda\bigr] t \le (\beta-1)t ,
\]
establishing the claimed inequality~\eqref{eq:thm:LB:EZ}.  
To get a handle on $\Var Z$ in~\eqref{eq:thm:LB:VZ}, note that~$Z$ is a restriction of~$X$ to a subset of the edges of $\cH$. 
So, with~\eqref{eq:thm:LB:constr:VarX} and~$\E(Y_eY_f) - \E Y_e \E Y_f \ge 0$ in mind, it is not difficult to see that $\Var Z \le \Var X$ holds. 
By Remark~\ref{rem:Var} there is a constant $A=A(k,a,D) >0$ such that 
\begin{equation}\label{eq:thm:LB:constr:VarX:UB}
\Var X \le A \Lambda = A\mu(1+np^{k-1}).
\end{equation}
Recalling $\mu \ge a n^2p^k$, it is easy to see that $\mu < 1/2$ implies $p = O(n^{-2/k})$ and $\Var X \le B = B(A,k,a,D) > 0$.  
Using~\eqref{eq:thm:LB:constr:VarX:UB} and the assumed lower bound for~$t$ in case of $\mu \ge 1/2$, 
it follows (exploiting $1 \le \lambda t$) that  
\[
\Var Z \le \Var X \le \indic{\mu < 1/2}B + \indic{\mu \ge 1/2} \max\{A,1\} t^2 \le \max\{B \lambda^2,A,1\} t^2 ,
\]
completing the proof. 
\end{proof}
Using a variant of the above proof, it alternatively suffices to assume $t \ge \max\{\sqrt{p\Var X},1\}$, say. 
Furthermore, for $p=o(1)$ and $t=O(\mu)$ with $t = \omega(1)$ we can easily improve the constant $c$ by planting only $(1+o(1))t$ edges (in some cases, this approach presumably yields the `optimal' form of the exponent).

\subsection{Configurations with many disjoint edges}\label{sec:LB:disj}
The second lower bound is based on the heuristic that, for small~$p$, most edges of~$\cH_p$ should arise disjointly. 
Exploiting the implied `approximate independence' of the edges, we obtain the following Chernoff-like lower bound. 
In fact, \eqref{eq:thm:LB:disj} is of sub-Gaussian type since $\E X = (1+o(1)) \Var X$ for the~$p$ under consideration. 
\begin{theorem}\label{thm:LB:disj}
Given $k \ge 3$, $a > 0$ and $D \ge 1$, 
let $\cH=\cH_n$ be a $k$-uniform hypergraph satisfying $v(\cH) \le Dn$, $e(\cH) \ge an^2$ and $\Delta_2(\cH) \le D$. 
Set $X=e(\cH_p)$, $\mu=\E X$ and $\varphi(x)=(1+x)\log(1+x)-x$.  
There are $n_0,c,d>0$ (depending only on $k,a,D$) such that for all $n \ge n_0$, $0 < p \le n^{-2/(k+1/3)}$ and $t \ge 0$ satisfying $1 \le \mu+t \le 9\max\{\mu,n^{1/(2k)}\}$ we have
\begin{equation}\label{eq:thm:LB:disj}
\Pr(X \ge \mu+t) \ge d \exp\Bigl(-c\varphi(t/\mu)\mu\Bigl) \ge d \exp\Bigl(-ct^2/\mu \Bigl) .  
\end{equation}
\end{theorem}
We have not tried to optimize $p \le n^{-2/(k+1/3)}$, but conjecture that this condition can be relaxed to $p=O(n^{-1/(k-1)})$. 
In fact, it would be interesting to have a general method which yields such Poisson-type lower bounds for the upper tail when $\Var X  = (1+o(1)) \E X$ holds (for the lower tail this was very recently settled by Janson and Warnke~\cite{JW}).  
In the proof of Theorem~\ref{thm:LB:disj} we shall use the idea that, for small~$p$, most edges $f \in \cH$ should appear disjointly (and thus nearly  independently) in $\cH_p$. 
The next lemma makes this more precise: it relates $\Pr(X=m)$ with $\Pr(\Bin(e(\cH),p^k)=m)$ over a convenient (but ad-hoc) range of~$m$. 
\begin{lemma}\label{lem:LB:disj}
Given $k \ge 3$, $a > 0$ and $D \ge 1$, 
let $\cH=\cH_n$ be a $k$-uniform hypergraph satisfying $v(\cH) \le Dn$, $e(\cH) \ge an^2$ and $\Delta_2(\cH) \le D$. 
Set $X=e(\cH_p)$ and $\mu=\E X$. 
There are $n_0,b>0$ (depending only on $k,a,D$) such that for all $n \ge n_0$, $0 < p \le n^{-2/(k+1/3)}$ and 
integers 
$0 \le m \le 99\max\{\mu,n^{1/(2k)}\}$ we have
\begin{equation}\label{eq:lem:LB:disj}
\Pr(X=m) \ge e^{-b} \binom{e(\cH)}{m} p^{km}(1-p^k)^{e(\cH)-m} . 
\end{equation}
\end{lemma}
With Lemma~\ref{lem:LB:disj} in hand, the proof of Theorem~\ref{thm:LB:disj} essentially reduces to folklore lower bounds for the binomial distribution (based on Stirling's formula); we include the details in Appendix~\ref{apx} for completeness (some minor care is needed when $t$ is small).  
A similar analysis can be used to tighten 
related results in the theory of random graphs due to DeMarco and Kahn~\cite{KkTailDK} and {\v{S}}ileikis~\cite{Matas2012}.

Let us informally discuss the strategy used in the proof of Lemma~\ref{lem:LB:disj}. 
For~\eqref{eq:lem:LB:disj} the basic plan is to consider the event that $\cH_p$ consists of \emph{exactly}~$m$ vertex disjoint edges. 
It turns out that, for small~$m$, there are roughly $\binom{e(\cH)}{m}$ ways to select such edge collections, and with probability~$p^{km}$ their $m$ disjoint edges are all present.  
Of course, we also need to take into account that all of the remaining $e(\cH)-m$ edges are \emph{not} present (to avoid overcounting). If these were independent events, then this would yield another factor of $(1-p^k)^{e(\cH)-m}$, and for small~$p$ we expect that this is usually close to the truth. 
The proof below follows the discussed outline, dropping the (de facto redundant) disjointness condition. 
However, we need to deal with one subtle technicality that we ignored so far: given a collection of edges $\{f_1, \ldots, f_{m}\} \subseteq \cH$, it can happen that the union of their vertex sets $\bigcup_{i \in [m]}f_i$ induces additional `extra' edges from~$\cH$ (even if all the~$f_i$ are vertex disjoint). 
In particular, for our construction this means that the second part is \emph{impossible}:  in this `bad' case at least one of the remaining~$e(\cH)-m$ edges must occur. 
Luckily, such bad edge collections are rare for small~$m$, so we can simply ignore them in our proof (see the definition of $\fS_m$ below). 
\begin{proof}[Proof of Lemma~\ref{lem:LB:disj}]
Define 
\begin{equation}\label{eq:lem:LB:disj:H}
\fS_{m} =\Bigl\{\cI \subseteq \cH: \text{ $e(\cI)=m$, \ and there are no $g \in \cH \setminus \cI$ with $g \subseteq \bigcup_{f \in \cI}f$}\Bigr\} .
\end{equation}
Recall that $f \in \cH_p$ if and only if $f \subseteq V_p(\cH)$. 
As the union of all edges in $\cI \in \fS_{m}$ contains at most~$km$ vertices, we have $\Pr(\cI \subseteq \cH_p) \ge p^{km}$ (for disjoint edges this would hold with equality.) 
So, since the events $\{\cI = \cH_p\}_{\cI \in \fS_{m}}$ are mutually exclusive, using $\Pr(\cI = \cH_p)=\Pr(\cI \subseteq \cH_p) \Pr(\cI = \cH_p \mid \cI \subseteq \cH_p)$ it follows that  
\begin{equation}\label{eq:lem:LB:disj:m}
\Pr(X=m) \ge \sum_{\cI \in \fS_{m}} \Pr(\cI \subseteq \cH_p) \Pr(\cI = \cH_p \mid \cI \subseteq \cH_p)  \ge |\fS_{m}| p^{km} \min_{\cI \in \fS_{m}}\Pr(\cI = \cH_p \mid \cI \subseteq \cH_p).
\end{equation}
It remains to estimate $|\fS_{m}|$ and $\Pr(\cI = \cH_p \mid \cI \subseteq \cH_p)$ from below. 
We defer the routine proof of the auxiliary claim that there is $\lambda=\lambda(k,a,D) >0$ such that for $n \ge n_0(k,a,D)$ we have 
\begin{equation}\label{eq:lem:LB:disj:ampt}
k^3D^3nm^2/e(\cH) \le 1/2 \quad \text{and} \quad  \max\bigl\{nm^3/e(\cH), \: m^2p, \: nmp^{k-1}\bigr\} \le \lambda .
\end{equation}

We bound $|\fS_{m}|$ from below by constructing certain edge-subsets $\cI=\{f_1, \ldots, f_m\} \in \fS_{m}$, counting the number of choices in each step. 
For $0 \le j < m$ we iteratively select $f_{j+1} \in \cH \setminus (\cB_{1,j+1} \cup \cB_{2,j+1})$, where 
\begin{equation*}
\begin{split}
\cB_{x,j+1}&=\bigl\{f \in \cH: \text{there is $g \in \cH$ with $|g \cap  \bigcup_{i \in [j]} f_{i}| \ge x$ and $|g \cap f| \ge 3-x$}\bigr\} .
\end{split}
\end{equation*}
Since $\{f_1, \ldots, f_j\} \subseteq \cB_{1,j+1}$ holds (consider $g=f=f_i$), all edges $f_i$ are distinct (in fact, vertex disjoint). 
Next, aiming at a contradiction, suppose there is an edge $g \in \cH \setminus \cI$ and an index $\ell \in [m]$ such that $g \subseteq \bigcup_{i \in [\ell]}f_i$ and $g \not\subseteq \bigcup_{i \in [\ell-1]}f_i$. 
If $|g \cap \bigcup_{i \in [\ell-1]} f_{i}| = 1$, then $|g \cap f_\ell| = k-1 \ge 2$ implies $f_{\ell} \in \cB_{1,\ell}$. 
If $|g \cap \bigcup_{i \in [\ell-1]} f_{i}| \ge 2$, then $|g \cap f_\ell| \ge 1$ implies $f_{\ell} \in \cB_{2,\ell}$. 
Both conclusions contradict $f_{\ell} \not\in (\cB_{1,\ell} \cup \cB_{2,\ell})$, 
showing that all constructed sets $\cI=\{f_1, \ldots, f_m\}$ indeed satisfy $\cI \in \fS_{m}$. 
Turning to the number of choices in the above greedy construction, note that $|\cB_{1,j+1}| \le kj \cdot \Delta_1(\cH) \cdot \binom{k}{2} \Delta_2(\cH)$ and $|\cB_{2,j+1}| \le \binom{kj}{2} \Delta_2(\cH) \cdot k \Delta_1(\cH)$. 
Since $\Delta_2(\cH) \le D$ and $\Delta_1(\cH) \le v(\cH) \Delta_2(\cH) \le D^2 n$, we infer that for each edge $f_{j+1}$ there are at least 
\[
e(\cH)-\bigl(|\cB_{1,j+1}|+|\cB_{2,j+1}|\bigr) \ge e(\cH)-\bigl(k^3D^3n j/2 + k^3D^2n j^2/2\bigr) \ge e(\cH)-k^3D^3n j^2
\]
choices. Recall that $1-x \ge e^{-2x}$ if $x \in [0,1/2]$. 
Since each edge-subset~$\cI$ can be generated in up to $m!$ different ways by our greedy construction, using $z^y/y! \ge \binom{z}{y}$ and \eqref{eq:lem:LB:disj:ampt} it follows for~$b=8k^3D^3 \lambda$ that, say,  
\begin{equation}\label{eq:lem:LB:disj:LB}
\begin{split}
|\fS_{m}| &\ge \frac{\prod_{0 \le j < m} \bigl(e(\cH)-k^3D^3nj^2\bigr)}{m!} \ge \frac{e(\cH)^m}{m!} \left(1-\frac{k^3D^3nm^2}{e(\cH)}\right)^m \\
& \ge \binom{e(\cH)}{m} \exp\bigl(-2k^3D^3nm^3/e(\cH)\bigr) \ge \binom{e(\cH)}{m} e^{-b/4} .
\end{split}
\end{equation}

Next, we estimate $\Pr(\cI = \cH_p \mid \cI \subseteq \cH_p)$ for all $\cI \in \fS_{m}$. 
Let $\cF_2$ contain all $g \in \cH \setminus \cI$ with $2 \le |g \cap \bigcup_{f \in \cI} f| < k$.
Similarly, let $\cF_1$ contain all $g \in \cH \setminus \cI$ with $|g \cap \bigcup_{f \in \cI} f| =1$. 
Set $\cF_0 = \cH \setminus (\cI \cup \cF_1 \cup \cF_2)$, and note that by definition of $\fS_{m}$, see \eqref{eq:lem:LB:disj:H}, all $g \in \cF_0$ satisfy $|g \cap \bigcup_{f \in \cI} f| =0$.  
Since $f \in \cH_p$ if and only if $f \subseteq V_p(\cH)$, using Harris' inequality~\cite{Harris1960} we deduce, say,  
\begin{equation*}\label{eq:lem:LB:disj:PP1}
\begin{split}
\Pr(\cI = \cH_p \mid \cI \subseteq \cH_p) & = 
\Pr\Bigl(\bigcap_{g \in \cF_0 \cup \cF_1 \cup \cF_2} \{g \not\subseteq V_p(\cH)\} \mid \bigcup_{f \in \cI} f \subseteq V_p(\cH)\Bigr)\\
&\ge (1-p^k)^{|\cF_0|} (1-p^{k-1})^{|\cF_1|}(1-p)^{|\cF_2|} . 
\end{split}
\end{equation*}
Note that $|\cF_1| \le k m \cdot \Delta_1(\cH)$ and $|\cF_2| \le \binom{km}{2} \cdot \Delta_2(\cH)$. 
Since $\Delta_1(\cH) \le D^2n$ and $\Delta_2(\cH) \le D$, 
using \eqref{eq:lem:LB:disj:ampt} we infer, by choice of~$b=8k^3D^3 \lambda$, that 
\[
|\cF_1|p^{k-1}+|\cF_2|p \le kD^2nmp^{k-1}+k^2Dm^2p \le (kD^2+k^2D)\lambda \le b/4 .
\]
Recalling $1-x \ge e^{-2x}$ if $x \in [0,1/2]$, using $p \le 1/2$ and $|\cF_0| \le e(\cH)-m$ we thus obtain 
\begin{equation*}\label{eq:lem:LB:disj:PP2}
\begin{split}
\Pr(\cI = \cH_p \mid \cI \subseteq \cH_p) &  \ge (1-p^k)^{|\cF_0|} e^{-2(|\cF_1|p^{k-1}+|\cF_2|p)} \ge (1-p^k)^{e(\cH)-m} e^{-b/2},
\end{split}
\end{equation*}
which together with \eqref{eq:lem:LB:disj:m} and \eqref{eq:lem:LB:disj:LB} establishes inequality~\eqref{eq:lem:LB:disj}, with room to spare.  

In the remainder we sketch the verification of~\eqref{eq:lem:LB:disj:ampt}, using the convention that all implicit constants may depend on~$k,a,D$. 
Let $\alpha =2/(k+1/3)$ and $\beta=2-k\alpha=2/(3k+1)=\alpha/3$, so that $\mu = O(n^2p^k)$, $p \le n^{-\alpha}$ and $1/(2k) \le \beta$ imply $m=O(n^{\beta})$. 
Using $e(\cH) = \Omega(n^2)$, $p \le n^{-\alpha}$ and $\beta < 1/3$ it now is routine to check that~\eqref{eq:lem:LB:disj:ampt} holds for suitable $\lambda>0$ 
(as $2\beta-1 < 0$ and $\max\{3\beta-1, \: 2\beta-\alpha, \: 1+\beta-(k-1)\alpha\} \le 0$ for $k \ge 3$). 
\end{proof}

\subsection{Configurations with too many vertices}\label{sec:LB:vx}
Our third lower bound is based on the following heuristic: if $V_p(\cH)$ contains `too many vertices' (more than expected), then it seems likely that the induced subgraph $\cH_p=\cH[V_p(\cH)]$ also contains `too many edges' (more than the average number). 
For moderately large~$p$, this approach eventually yields the following lower bound of sub-Gaussian type  
(by Remark~\ref{rem:Var} we have $\Lambda = \Theta(\Var X)$, since $p$ is bounded away from one). 
\begin{theorem}\label{thm:LB:vx}
Given $k \ge 2$, $a > 0$ and $D \ge 1$, 
let $\cH=\cH_n$ be a $k$-uniform hypergraph satisfying $v(\cH) \le Dn$, $e(\cH) \ge an^2$ and $\Delta_2(\cH) \le D$. 
Set $X=e(\cH_p)$, $\mu=\E X$, $\Lambda = \mu (1+np^{k-1})$ and $\varphi(x)=(1+x)\log(1+x)-x$.  
Given $\alpha \in (0,1)$, there are $n_0>0$ (depending only on $k,a,D$) and $\beta,c>0$ (depending only on $\alpha,k,a,D$) such that for all $n \ge n_0$, $\alpha n^{-1/(k-1)}\le p \le 1-\alpha$ and $\min\{\sqrt{\Lambda},\sqrt{\Var X}\} \le t \le \beta \mu$ we have
\begin{equation}\label{eq:thm:LB:vx}
\Pr(X \ge \mu+t) \ge \exp\Bigl(-c \varphi(t/\mu)\mu^2/\Lambda\Bigl) \ge \exp\Bigl(-ct^2/\Lambda\Bigl) .  
\end{equation}
\end{theorem}
The key observation is that $\mu^2/\Lambda = \Theta(np)$ for the relevant range of $p$. 
With this in mind, the proof of Theorem~\ref{thm:LB:vx} is based on the following two ideas: (i)~since $V_p(\cH) \sim \Bin(v(\cH),p)$ and $v(\cH)=\Theta(n)$, with probability at least $\exp(-\Theta(\eps^2 np))= \exp(-\Theta((\eps\mu)^2/\Lambda))$ we have $|V_p(\cH)| \ge (1+\eps)\E |V_p(\cH)|$, and (ii)~conditioning on $|V_p(\cH)| \ge (1+\eps)\E |V_p(\cH)|$ intuitively increases the expected number $e(\cH_p)=e(\cH[V_p(\cH)])$ of induced edges, effectively turning the unlikely event $X \ge \mu+t$ into a `typical' one; see also~\eqref{eq:heur:PZ} below. 
For the number of copies of $H$ in the binomial random graph $G_{n,p}$ an analogous reasoning (based on a deviation of the number of edges) applies for $p = \Omega(n^{-1/m_2(H)})$, where $m_2(H)$ is the so-called $2$-density of $H$; for the lower tail this idea was used by Janson and Warnke~\cite{JW}. 

We now informally discuss the high-level structure of the proof, which is similar to~Theorem~\ref{thm:LB:constr}. 
Let~$\mu=\E X$, $\eps=t/\mu$ and $m=(1+\eps)\E |V_p(\cH)|$. Applying (i) as outlined above, using monotonicity we expect that  
\begin{equation*}
\begin{split}
\Pr(X \ge \mu+t) & \;\ge\; \Pr(|V_p(\cH)| \ge m) \cdot \Pr\bigl(X \ge \mu+t \: \big| \: |V_p(\cH)| \ge m\bigl) \\
& \;\ge\; e^{-\Theta(t^2/\Lambda)} \cdot \Pr\bigl(X \ge \mu+t \: \mid \: |V_p(\cH)| = m\bigr)  .
\end{split}
\end{equation*}
Thinking of the uniform random graph $G_{n,m}$, using $\E |V_p(\cH)|=v(\cH) p$ it seems plausible that $\E(X \mid |V_p(\cH)| = m)$ is approximately $e(\cH) \cdot \bigl(m/v(\cH)\bigr)^k = (1+\eps)^k \E X$. 
Similarly, we expect $\Var(X \mid |V_p(\cH)| = m) = O(\Var X)$ for $\eps=O(1)$. 
Noting $t=\eps \E X$ and $(1+\eps)^k > 1+2\eps$, we see that $\E(X \mid |V_p(\cH)| = m)-t$ ought to be roughly at least $(1+\eps)\E X=\mu+t$. 
To sum up, for $t \ge \sqrt{\Var X}$ the Paley--Zygmund inequality~\eqref{eq:PZ} should yield 
\begin{equation}\label{eq:heur:PZ}
\begin{split}
\Pr\bigl(X \ge \mu+t \: \big| \: |V_p(\cH)| = m\bigr) & \;\ge\; \Pr\Bigl(X \ge \E\bigl(X \: \big| \: |V_p(\cH)| = m\bigr)-t \ \Big| \ |V_p(\cH)| = m\Bigr) \\
& \;\ge\; \Omega\left(\frac{t^2}{\Var X+t^2}\right)  = \Omega(1) ,
\end{split}
\end{equation}
and the following proof basically makes this rigorous (with some care about border cases). 
\begin{proof}[Proof of Theorem~\ref{thm:LB:vx}]
Let~$\eps=t/\mu$, $N=v(\cH)$, and~$m=(1+\eps)Np$. 
Given $0 \le j \le N$, we henceforth write $\Pr_j(\cdot)=\Pr(\cdot \mid |V_p(\cH)|=j)$ for brevity. We analogously use $\E_j(\cdot)$ and $\Var_j(\cdot)$, respectively. 
Note that, by monotonicity, we have  
\begin{equation}\label{eq:LB:vx:P}
\Pr(X \ge \mu+t) \ge \sum_{j \ge m} \Pr_j(X \ge \mu+t)\Pr(|V_p(\cH)|=j) \ge \Pr_m(X \ge \mu+t) \Pr(|V_p(\cH)| \ge m) .
\end{equation}

It remains to estimate $\Pr_m(X \ge \mu+t)$ and $\Pr(|V_p(\cH)| \ge m)$ from below. 
We start by defining~$\beta=\beta(\alpha,k,a,D) \in (0,1)$ in a somewhat technical way (that will be convenient in border cases). 
We use the convention that all implicit constants may depend on~$k,a,D$ (but not on~$\alpha$). 
In particular, $e(\cH) =\Omega(n^2)$ and $\Delta_2(\cH)=O(1)$ imply $v(\cH) = \Omega(n)$, so that $N=\Theta(n)$. 
Observing that $\Lambda Np/\mu^2 = \Theta(1+(np^{k-1})^{-1})$ holds, we infer  
\begin{equation}\label{eq:LB:vx:Cond}
\eps^2 Np = \Omega(\eps^2\mu^2/\Lambda) \quad \text{ and } \quad \eps^2 Np = O\bigl((1+\alpha^{-(k-1)}) \eps^2\mu^2/\Lambda\bigr) . 
\end{equation}
Furthermore, by assumption and Remark~\ref{rem:Var} we have $\eps \mu =t \ge \min\{\sqrt{\Lambda},\sqrt{\Var X}\} = \Omega(\sqrt{\alpha\Lambda})$, so that 
$\eps^2 Np = \Omega(\alpha)$ by~\eqref{eq:LB:vx:Cond}. 
With $\eps \le \beta$ in mind, we now pick $\beta \in (0,\alpha/4]$ small enough such that 
\begin{equation}\label{eq:LB:vx:Cond:beta}
\eps Np = \eps^2 Np/\eps = \Omega(\alpha\beta^{-1}) \ge 2k^2 \quad \text{ and } \quad Np = \Omega(\alpha\beta^{-2}) \ge 16 \alpha^{-2} .
\end{equation}
Note that $m=(1+\eps) Np \le (1+\alpha)(1-\alpha)N < N$. 
So, since $N=\Theta(n)$ and $|V_p(\cH)| \sim \Bin(N,p)$,
for $n \ge n_0(k,a,D)$ folklore estimates for binomial random variables yield 
\begin{equation}\label{eq:LB:vx:Chernoff}
\Pr(|V_p(\cH)| \ge m) = \Pr(|V_p(\cH)| \ge (1+\eps) Np) \ge d_1\exp\bigl(-c_1 \eps^2Np)\bigr) ,
\end{equation}
where the constants $c_1,d_1>0$ depend only on~$\alpha,k,a,D$. 
(This can, e.g., be deduced analogous to the proof of Theorem~\ref{thm:LB:disj} by means of Stirling's formula. 
One minor difference in the estimates is perhaps that in \eqref{eq:Bin:LB} we can, e.g., via $1-q =1-p \ge \alpha$ and $j \le 4T =4 \max\{\eps Np,\sqrt{Np}\}$ here directly obtain $j^2/\bigl((1-q)N\bigr) = O(\alpha^{-1}\eps^2Np + \alpha^{-1})$, say. 
To be pedantic, by choice of~$\beta$ in~\eqref{eq:LB:vx:Cond:beta} we have also ensured that $M \le Np + 4 T = Np(1+4\max\{\eps,1/\sqrt{Np}\}) \le N (1-\alpha)(1+\alpha) < N$ holds.)

Turning to $\Pr_m(X \ge \mu+t)$, note that $\eps \le \beta \le 1$ implies $\varphi(\eps)=\Theta(\eps^2)$ via \eqref{eq:varphi:x4} and \eqref{eq:varphi:x3}. 
So, in view of~\eqref{eq:LB:vx:Cond}, \eqref{eq:LB:vx:Chernoff} and $\eps=t/\mu$, we see that \eqref{eq:thm:LB:vx} follows if $\Pr_m(X \ge \mu+t) \ge d_2=d_2(\alpha,k,a,D)>0$.  
Define $I_{f}=\indic{f \subseteq V_p(\cH)}$, so that $X=\sum_{f \in \cH}I_f$. 
Let $V_m(\cH) \subseteq V(\cH)$ with $|V_m(\cH)|=m$ be chosen uniformly at random. 
Observe that $V_p(\cH)$ conditioned on $|V_p(\cH)|=m$ has the same distribution as $V_m(\cH)$. 
Using $m =(1+\eps) Np \ge 2k^2$, $|f|=k \ge 2$ and $\E I_f = p^k$ it follows that 
\[
\E_m(I_f) = \binom{N-k}{m-k} \Big/ \binom{N}{m} = \prod_{0 \le i < k} \frac{m-i}{N-i} \ge (1-k/m)^k (m/N)^k  \ge (1-k^2/m) (1+\eps)^2 \E I_f  .
\]
Hence $\E_m(X) \ge (1-k^2/m) (1+\eps)^2 \E X$. 
Furthermore, by \eqref{eq:LB:vx:Cond:beta} we have $m=(1+\eps)Np \ge 2k^2 (1+\eps)\eps^{-1}$, which implies $(1-k^2/m)(1+\eps) \ge 1+\eps/2$. 
So, recalling $t=\eps \mu = \eps \E X$, we obtain  
\begin{equation}\label{eq:LB:vx:E}
\E_m(X) - t/2 \ge (1+\eps/2)(1+\eps)\E X - \eps \E X/2  \ge (1+\eps) \E X = \mu + t.
\end{equation}
Similar standard calculations (see, e.g., the proof of Theorem~15 in~\cite{JW}) show that, say,  
\begin{equation}\label{eq:LB:vx:Var}
\Var_m(X) = \sum_{(e,f) \in \cH \times \cH} \bigl[\E_m(I_eI_f) - \E_m(I_e)  \E_m(I_f) \bigr] \le (1+\eps)^{2k}\sum_{(e,f) \in \cH \times \cH: e \cap f \neq \emptyset} \E(I_eI_f) .
\end{equation}
It is not difficult to see that the final expression of~\eqref{eq:LB:vx:Var} is at most $4^k \cdot O(\Lambda)$, 
so that Remark~\ref{rem:Var} and $1-p \ge \alpha$ imply $\Var_m(X) = O(\alpha^{-1} \min\{\Lambda,\Var X\})$, say. 
Using the assumed lower bounds for~$t$, we now infer $\Var_m(X) =  O(\alpha^{-1}t^2)$. 
Recalling~\eqref{eq:LB:vx:E}, the Paley--Zygmund inequality~\eqref{eq:PZ} implies 
\[
\Pr_m(X \ge \mu + t) \ge \Pr_m(X \ge \E_m(X) - t/2) \ge \frac{(t/2)^2}{\Var_m(X)+(t/2)^2} = \Omega\left(\frac{1}{\alpha^{-1}+1}\right) ,
\]
which, as discussed, completes the proof. 
\end{proof}

\subsection{Proof of the lower bounds of Theorem~\ref{thm:HGp} and~\ref{thm:HG} (and Remark~\ref{rem:HG})}\label{sec:HGproofLB}
In this section we combine the previous estimates, and prove the 
lower bounds of Theorem~\ref{thm:HGp} and~\ref{thm:HG} (as well as Remark~\ref{rem:HG}). 
This is in principle straightforward but, at least as written here, requires several case distinctions (that are not very illuminating). 
Some complications are due to the fact that the results of Sections~\ref{sec:LB:constr}--\ref{sec:LB:vx} are only valid in some range of the parameters (they need to be merged seamlessly), whereas others stem from the fact that  our estimates are uniform (e.g.,~our~$n_0$ does \emph{not} depend on~$\eps$ or~$\gamma$), from the fact that our assumptions are very weak (e.g., $p>0$ instead of $p \ge n^{-2/k}$), or from the fact that the exponents are more involved than usual (e.g., \eqref{eq:thm:HG:LB} yields up to five different asymptotic expressions). 
\begin{proof}[Proof of \eqref{eq:thm:HGp:LB} of Theorem~\ref{thm:HGp}]
The case $\sqrt{\mu}\log(1/p) \le \mu$ is easy: then Theorem~\ref{thm:LB:constrp} implies  
\begin{equation}\label{eq:thm:proofLBp:1}
\Pr(X \ge (1+\eps)\mu) \ge \exp\bigl(-2D\max\{1,\sqrt{\eps}\} \sqrt{\mu} \log(1/p)\bigr) . 
\end{equation}
In the remainder we may thus assume $\sqrt{\mu}\log(1/p) \ge \mu$, which for $n \ge n_0(k,a)$ implies $p \le n^{-2/(k+1/3)}$, with room to spare. 
If $\eps \mu \le \max\{\mu,n^{1/(2k)}\}$, then Theorem~\ref{thm:LB:disj} and $1 \le 2 \max\{\mu,\eps\mu\}$ (as $(1+\eps)\mu \ge 1$) yield 
\begin{equation}\label{eq:thm:proofLBp:2}
\Pr(X \ge (1+\eps)\mu) \ge \exp\bigl(-\log(1/d)-c\eps^2\mu\bigr) \ge \exp\bigl(-2 \max\{2\log(1/d),c\}\max\{1,\eps^2\}\mu\bigr) . 
\end{equation}
It remains to consider the case $\eps \mu \ge \max\{\mu,n^{1/(2k)}\}$. 
Since $p \log(1/p) \le 1$ analogous to~\eqref{eq:ps}, using $\mu \le D^3 n^2 p^k$ and $p \le n^{-2/(k+1/3)}$ it follows for $n \ge n_0(k,D)$ that   
\begin{equation*}\label{eq:cd}
\sqrt{\mu}\log(1/p) \le \indic{p \le n^{-4/(k-2)}}D^{3/2}n p^{(k-2)/2} \cdot p \log(1/p)+ \indic{p \ge n^{-4/(k-2)}}4D^{3/2} n p^{k/2} \log(n) \le n^{1/(2k)} \le \eps \mu .
\end{equation*}
Since $\sqrt{1+\eps} \le 2\eps$ (as $\eps \mu \ge \mu$ implies $\eps \ge 1$), now Theorem~\ref{thm:LB:constrp} gives  
\begin{equation}\label{eq:thm:proofLBp:3}
\Pr(X \ge (1+\eps)\mu) \ge \exp\bigl(-D \sqrt{(1+\eps)\mu} \log(1/p)\bigr) \ge \exp\bigl(-2D \eps^2 \mu \bigr) .
\end{equation}
To sum up, \eqref{eq:thm:proofLBp:1}--\eqref{eq:thm:proofLBp:3} readily establish the lower bound~\eqref{eq:thm:HGp:LB}, completing the proof.  
\end{proof}
\begin{proof}[Proof of \eqref{eq:thm:HG:LB} of Theorem~\ref{thm:HG} and Remark~\ref{rem:HG}]
Note that we may assume $\gamma \le 1/2$ (since decreasing $\gamma$ yields less restrictive assumptions). 
We use the convention that all implicit constants may depend on $k,a,D$ (not on~$\gamma$), and tacitly assume $n \ge n_0(k,a,D)$ whenever necessary. 
With foresight, we start with some technical but useful auxiliary estimates. 
Recalling~\eqref{eq:varphi:x4}, for $t=\beta \mu$ we have $\varphi(t/\mu)\mu^2 \ge \min\{\beta,\beta^2\}\mu^2/3$. 
Since $\mu = \Theta(n^2p^k)$ and $\Lambda = \mu(1+np^{k-1})$, it follows for $t=\beta \mu$ that 
\begin{equation}\label{eq:ie}
\begin{split}
& \frac{\varphi(t/\mu)\mu^2}{\sqrt{t}\log(1/p)\Lambda} \ge  \frac{\min\bigl\{\beta^{1/2},\beta^{3/2}\bigr\} \mu^{1/2}}{3(1+np^{k-1}) \log(1/p)}\\
& \quad = \min\bigl\{\beta^{1/2},\beta^{3/2}\bigr\}\left(\indic{p < n^{-1/(k-1)}} \frac{\Omega(np^{k/2})}{\log(1/p)}+\indic{p \ge n^{-1/(k-1)}}\frac{\Omega(1)}{p^{k/2-1}\log(1/p)}\right) .
\end{split}
\end{equation}
Analogously to~\eqref{eq:ps}, calculus yields $p^{1/2}\log(1/p) \in (0,2]$ for $p \in (0,1)$.
Since $k \ge 3$ entails $p^{k/2-1} \le p^{1/2}$, we see that $\gamma n^{-2/k}(\log n)^{2/k} \le p \le 1-\gamma$ and $t \ge \mu$ imply $C_1 \sqrt{t}\log(1/p) \le \varphi(t/\mu)\mu^2/\Lambda$, where $C_1=C_1(\gamma,k,a,D)>0$. 
Replacing $\log(1/p)$ with $\log(e/p)$ in~\eqref{eq:ie}, we similarly see that $C_2 \sqrt{t}\log(e/p) \le \varphi(t/\mu)\mu^2/\Lambda$ for all $\gamma n^{-2/k}(\log n)^{2/k} \le p \le 1$ and $t \ge \mu$, where $C_2=C_2(\gamma,k,a,D)>0$. 
Since \eqref{eq:varphi:x4} and \eqref{eq:varphi:x3} imply $\varphi(t/\mu)\mu^2 = \Theta(t^2)$ for $t \le \mu$, this completes the proof of Remark~\ref{rem:HG} (by adjusting the constants $n_0,c,C$). 

We turn to~\eqref{eq:thm:HG:LB} of Theorem~\ref{thm:HG}, and start with case~(iii), where $\gamma n^{-1/(k-1)} \le p \le 1-\gamma$. Applying Theorem~\ref{thm:LB:constr} and~\ref{thm:LB:vx} (with $\alpha = \gamma$) there is $\beta=\beta(\gamma,k,a,D)>0$ such that 
\begin{equation}\label{eq:thm:proofLB:case4}
\Pr(X \ge \mu+t) \ge \max\left\{ \exp\Bigl(-c_1 \sqrt{t}\log(1/p)\Bigr), \: \indic{t \le \beta \mu}\exp\Bigl(-c_2 \varphi(t/\mu)\mu^2/\Lambda\Bigl) \right\} .
\end{equation}
Proceeding as in the discussion following \eqref{eq:ie}, for $t \ge \beta \mu$ we infer $A \sqrt{t}\log(1/p) \le \varphi(t/\mu)\mu^2/\Lambda$, where $A=A(\beta,\gamma,a,k,D)>0$. 
Replacing $c_2$ by $c_3=\max\{c_2,c_1/A\}$ we thus can remove the indicator $\indic{t \le \beta \mu}$ in~\eqref{eq:thm:proofLB:case4}, establishing~\eqref{eq:thm:HG:LB}.  

Next we consider case~(ii) in the range $n^{-1/(k-1)} \le p \le n^{-1/(k-1)} \log n$. 
As in~\eqref{eq:Var:b}, by Remark~\ref{rem:Var} we have $\Var X \ge  b \Lambda \ge b \mu$, where $b=b(k,a,D) \in (0,1]$. 
Since $\Lambda = O(\mu (\log n)^{k-1})$ and $\mu= \Omega(n^{(k-2)/(k-1)})$, 
it is easy to see that $t \ge b^{2/3}\mu^{2/3}(\log n)^{2/3} \ge \sqrt{\Lambda}$ holds. 
Hence, by case~(iii) above there is nothing to show. 

We now turn to case~(i), where $p \le n^{-2/(k+1/3)}$. 
If $\sqrt{t} \log(1/p) \le \varphi(t/\mu) \mu^2/\Lambda$ holds, then using $\varphi(t/\mu) \mu^2 \le t^2$, see~\eqref{eq:varphi:x3}, and $\Lambda=\Theta(\mu)$ we infer $t \ge \Lambda^{2/3} (\log(1/p))^{2/3} \ge \indic{\mu \ge 1/2}\sqrt{\Lambda}$, so Theorem~\ref{thm:LB:constr} applies.  
Noting $\mu^2/\Lambda = \Theta(\mu)$, it thus remains to show that Theorem~\ref{thm:LB:disj} applies when $\varphi(t/\mu) \mu^2/\Lambda \le \sqrt{t} \log(1/p)$. 
Aiming at a contradiction, we now assume that $t \ge 8\max\{\mu,n^{1/(2k)}\}$.  
Noting that $\varphi(x)= (1+x)\log(1+x)-x \ge x (\log x) /2$ for $x \ge e^2 \approx 7.4$, using $\Lambda = \Theta(\mu)$ we infer
\begin{equation}\label{eq:thm:proofLB:case2:smallp}
1 \ge \frac{\varphi(t/\mu) \mu^2}{\sqrt{t} \log(1/p) \Lambda} \ge \frac{t^{1/2} \mu \log(t/\mu)}{2\log (1/p)\Lambda} = n^{1/(4k)} \cdot \Omega\biggl(\frac{\log(t/\mu)}{\log(1/p)}\biggr) .
\end{equation}
We now argue that the right hand side of~\eqref{eq:thm:proofLB:case2:smallp} is~$\omega(1)$. 
Observe that $p \le n^{-2/(k-1)}$ implies $t/\mu \ge \Omega(n^{1/(2k)}/(n^2p^k)\bigr) = \omega(p^{-1})$, and that $p \ge n^{-2/(k-1)}$ implies 
$\log(t/\mu)/\log(1/p) \ge \Omega((\log n)^{-1})$. 
In both cases we readily obtain a contradiction in \eqref{eq:thm:proofLB:case2:smallp} for large~$n$, which by our above discussion establishes~\eqref{eq:thm:HG:LB}. 

Finally, by case~(i) above it remains to verify case~(ii) in the range $n^{-2/(k+1/3)} \le p \le n^{-1/(k-1)}$. 
Note that $\Lambda = \Theta(\mu)$, $\Var X \ge  b \Lambda \ge b \mu$, and $\mu = \Omega(n^{2/(k+1)})$ imply $t \ge b^{2/3}\mu^{2/3}(\log n)^{2/3} \ge \sqrt{\Lambda}$ and $\mu+t \ge 1$, with room to spare.  
In case of $t \le \mu$, by \eqref{eq:varphi:x4} we have $\varphi(t/\mu) \mu^2 \ge t^2/3$, so that $\Lambda = \Theta(\mu)$ yields 
\[
\frac{\varphi(t/\mu)\mu^2}{\sqrt{t} \log(1/p) \Lambda} \ge \frac{t^{3/2}}{3\log(1/p) \Lambda} \ge \frac{b \mu \log n}{3\log(1/p) \Lambda}= \Omega(1) . 
\]
Using the discussion after \eqref{eq:ie} in case of~$t \ge \mu$, it thus follows (in both cases) that $B \sqrt{t} \log(1/p) \le \varphi(t/\mu)\mu^2/\Lambda$, where $B=B(b,\gamma,k,a,D)>0$. Hence an application of Theorem~\ref{thm:LB:constr} establishes~\eqref{eq:thm:HG:LB}. 
\end{proof}

\bigskip{\bf Acknowledgements.} 
I would like to thank Oliver Riordan and Matas {{\v{S}}ileikis for many useful remarks on an earlier version of this paper, and Svante Janson for a helpful discussion.  
I am also grateful to the referee for an exceptionally careful reading, and for numerous constructive suggestions concerning the presentation.

\small
\begin{spacing}{0.4}
\bibliographystyle{plain}

\end{spacing}
\normalsize

\begin{appendix}
\section{Appendix}\label{apx}
The following proof is based on Stirling's approximation formula $1 \le x!/[\sqrt{2\pi x}(x/e)^x ] \le e^{1/(12x)}$. 
Some of the minor complications below stem from the fact that our assumption $\mu+t \ge 1$ is extremely weak. 
\begin{proof}[Proof of Theorem~\ref{thm:LB:disj}]
With foresight, let $T=\max\{t,\sqrt{\mu}\}$, $L=\ceil{\mu+T}$ and $M=\ceil{\mu+2T}$. Clearly,   
\begin{equation}\label{eq:thm:LB:disj:0}
\Pr(X \ge \mu+t) \ge \Pr(X \ge \mu + T) \ge \sum_{m \in \NN: L \le m \le M}\Pr(X=m) . 
\end{equation}
In view of Lemma~\ref{lem:LB:disj}, we now estimate the right hand side of \eqref{eq:lem:LB:disj}. 
To avoid clutter, let~$N=e(\cH)$ and~$q=p^k$. 
Recalling $1-x \le e^{-x}$, $\mu=Nq>0$ and Stirling's formula, standard (somewhat tedious but simple) calculations show that for any $\mu+j \in \NN$ satisfying $1 \le \mu+j < N$ we have, say, 
\begin{equation}\label{eq:Bin:LB}
\begin{split}
\binom{N}{\mu+j}q^{\mu+j}(1-q)^{N-\mu-j} 
& \ge \frac{\exp\left(-\frac{1}{12(\mu+j)}-\frac{1}{12(N-\mu-j)}\right)}{\sqrt{2\pi (\mu+j)(1-q-\frac{j}{N}})\left(1+\frac{j}{\mu}\right)^{\mu+j}\left(1-\frac{j}{N-\mu}\right)^{N-\mu-j}} \\
& \ge \frac{\exp\left(-\frac{1}{6}-\bigl((\mu+j)\log(1+j/\mu)-j\bigr)-\frac{j^2}{(1-q)N}\right)}{\sqrt{2\pi (\mu+j)}} .
\end{split}
\end{equation}
Note that $(\mu+j)\log(1+j/\mu)-j=\varphi(j/\mu)\mu$, and that $\varphi(j/\mu)$ is monotone increasing in $j \ge 0$. 
Since $\mu+t \ge 1$ implies $T \ge 1/2$, we deduce $M - \mu \le 2T+1 \le 4T$. 
Since $N = e(\cH) \ge an^2$, from the proof of Lemma~\ref{lem:LB:disj} it follows that $M \le \mu + 4T=O(n^{\beta})$ satisfies $M^2/N = o(1)$ and $M < N$. 
In particular, $q =p^k\le 1/2$ implies $j^2/\bigl((1-q)N\bigr) \le 2 M^2/N = o(1)$. 
By combining \eqref{eq:thm:LB:disj:0} with Lemma~\ref{lem:LB:disj} and \eqref{eq:Bin:LB}, we now infer that, say,  
\[
\Pr(X \ge \mu+t) \ge \floor{\max\{T,1\}} \cdot \frac{\exp\Bigl(-(b+1)- \varphi(4T/\mu)\mu\Bigr)}{\sqrt{2\pi M}} .
\]
Noting that $\max\{T,1\} = \max\{t,\sqrt{\mu},1\}$ and $M \le 4 \max\{t,\mu,1\}$, we deduce $\max\{T,1\}/\sqrt{M} \ge 1/\sqrt{4}$.  
Next we estimate $\varphi(4T/\mu)\mu$. 
If $T=\sqrt{\mu}$ holds, then $\varphi(4T/\mu)\mu \le 16 T^2/\mu = 16$ by \eqref{eq:varphi:x3}, and if $T=t$ holds, then $\varphi(4T/\mu)\mu \le 16 \varphi(t/\mu)\mu$ by applying~\eqref{eq:varphi:x2} twice. 
Combining our findings, it follows that, say, 
\[
\Pr(X \ge \mu+t) \ge e^{-(b+17)} / \sqrt{32\pi} \cdot \exp\Bigl(-\indic{t > \sqrt{\mu}}16 \varphi(t/\mu)\mu\Bigr) ,
\]
which together with \eqref{eq:varphi:x3} readily establishes \eqref{eq:thm:LB:disj} with $c=16$ and $d=e^{-(b+17)} / \sqrt{32\pi}$.  
\end{proof}
\end{appendix}
\end{document}